\documentclass[a4paper,11pt]{article}

\usepackage{amsmath, amssymb, amsthm}
\usepackage{graphics}

\usepackage{psfrag, booktabs}
\usepackage{color, epsfig, float}
\usepackage[active]{srcltx}
\usepackage[bookmarksopen, colorlinks, linkcolor = blue, urlcolor = red,
            citecolor = red, menucolor = blue]{hyperref}

\usepackage{float}
\floatstyle{plain}
\newfloat{algorithm}{thp}{alg}
\floatname{algorithm}{Algorithm}

\newtheorem{theorem}{Theorem}[section]
\newtheorem{lemma}[theorem]{Lemma}
\newtheorem{remark}[theorem]{Remark}
\newtheorem{corollary}[theorem]{Corollary}
\newtheorem{proposition}[theorem]{Proposition}
\newtheorem{definition}[theorem]{Definition}

\newcommand\R{\mathbb{R}}

\newcommand\N{\mathbb{N}}

\DeclareMathOperator*{\argmin}{arg\, min}

\newcommand\norm[1]{\|#1\|}

\newcommand\set[1]{\{#1\}}

\begin{document}

\title{Range-relaxed criteria for choosing the Lagrange multipliers
       in the Levenberg-Marquardt method}

\setcounter{footnote}1

\author{
A.~Leit\~ao%
\thanks{Department of Mathematics, Federal University of St.\,Catarina,
        P.O.\,Box 476, 88040-900 Florian\'opolis, Brazil,
        \href{mailto:acgleitao@gmail.com}{\tt acgleitao@gmail.com},\,
        \href{mailto:fabiomarg@gmail.com}{\tt fabiomarg@gmail.com}.}
\and
F.~Margotti$^\dag$
\and
B.~F.~Svaiter%
\thanks{IMPA, Estr.\,Dona Castorina 110, 22460-320 Rio de Janeiro, Brazil,
        \href{mailto:benar@impa.br}{\tt benar@impa.br}.} }
\date{\small \today}

\maketitle

\begin{abstract}
In this article we propose a novel strategy for choosing the Lagrange multipliers in
the Levenberg-Marquardt method for solving ill-posed problems modeled by nonlinear
operators acting between Hilbert spaces.
Convergence analysis results are established for the proposed method, including:
monotonicity of iteration error, geometrical decay of the residual, convergence
for exact data, stability and semi-convergence for noisy data.
Numerical experiments are presented for an elliptic parameter identification
two-dimensional EIT problem. The performance of our strategy is compared with
standard implementations of the Levenberg-Marquardt method (using {\em a priori}
choice of the multipliers).
\end{abstract}

\medskip


\section{Introduction} \label{sec:intro}

In this article we address the Levenberg-Marquardt (LM) method  \cite{Le44, Ma63},
which is a well established iterative method for obtaining stable approximate
solutions of nonlinear ill-posed operator equations \cite{DES98, Ha97} (see also
the textbooks \cite{EngHanNeu96, KalNeuSch08} and the references therein).

The novelty of our approach consists in adopting a range-relaxed criteria for the
choice of the Lagrange multipliers in the LM method.
Our approach is inspired in the recent paper \cite{BLS19}, where a range-relaxed
criteria was proposed for choosing the Lagrange multipliers in the iterated Tikhonov
method for linear ill-posed problems.

With our strategy, the new iterate is obtained as the projection of the current
one onto a level-set of the linearized residual function. This level belongs to
an interval (or \emph{range}), which is defined by the current (nonlinear)
residual and by the noise level.
As a consequence, the admissible Lagrange multipliers (in each iteration)
shall belong to a non-degenerate interval instead of being a single value
(see \eqref{def:alphak-lm}).
This fact reduces the computational burden of evaluating the multipliers.
Moreover, under appropriate assumptions, the choice of the above mentioned
range enforces geometrical decay of the residual (see \eqref{eq:residual_est}).

The resulting method (see Section~\ref{sec:rrLM-method}) proves, in the preliminary
numerical experiments (see Section~\ref{sec:numeric}), to be more efficient than
the classical geometrical choice of the Lagrange multipliers, typically used in
implementations of LM type methods.

\subsection{The model problem}

The \emph{exact case} of the \textit{inverse problem} we are interested in
consists of determining an unknown quantity $x \in X$ from the set of data
$y \in Y$, where $X$, $Y$ are Hilbert spaces, and $y$ is obtained by indirect
measurements of the parameter $x$, this process being described by the model
\begin{equation}\label{eq:inv-probl}
    F(x)  =  y\, ,
\end{equation}
with $F: D(F) \subseteq X \to Y$ being a non-linear ill-posed operator.
In practical situations, one does not know the data exactly. Instead,
an approximate measured data $y^\delta \in Y$ satisfying
\begin{equation}\label{eq:noisy-d}
    \norm{ y^\delta - y } \le \delta \, ,
\end{equation}
is available, where $\delta > 0$ is the (known) noise level.

Standard methods for finding a solution of \eqref{eq:inv-probl} are
based in the use of \textit{Iterative type} regularization methods
\cite{BakKok04, EngHanNeu96, HanNeuSch95, KalNeuSch08, Lan51},
which include the LM method, or \textit{Tikhonov type} regularization methods
\cite{EngHanNeu96, Mor93, SeiVog89, Tik63b, TikArs77, Sch93a}.

\subsection{The Levenberg-Marquardt method}

In what follows we briefly revise the LM method, which was proposed separately
by K.\,Levenberg \cite{Le44} and D.W.\,Marquardt \cite{Ma63} for solving nonlinear
optimization problems.
The LM method for solving the nonlinear ill-posed operator equation
\eqref{eq:inv-probl} was originally considered in \cite{DES98, Ha97},
and is defined by
\begin{align*}
  x_{k+1}^\delta \ := \ \argmin \ \big\{
  \| y^\delta - F(x_k^\delta) - F'(x_k^\delta) (x - x_k^\delta) \|^2
  + \alpha_k \| x - x_k^\delta \|^2 \big\} \, , \ k = 0, 1, \dots
\end{align*}
%
%
Here $F'(z): X \to Y$ is the Fr\'echet-derivative of $F$ in $z \in D(F)$,
$F'(z)^*: Y \to X$ is the corresponding adjoint operator and
$x_0^\delta \in X$ is some initial guess (possibly incorporating \textit{a priori}
knowledge about the exact solution(s) of $F(x)=y$).
Moreover, $\{ \alpha_k \}$ is a sequence of positive relaxation parameters
(or Lagrange multipliers), aiming to guarantee convergence and stability
of the iteration. This method can be summarized as follows
\begin{equation}
 \label{def:lm}
x_{k+1}^\delta  =  x_{k}^\delta + h_k \, , \ \ {\rm whith} \ \ \
          h_k  :=  \big( F'(x_k^\delta)^* F'(x_k^\delta) + \alpha_k I \big)^{-1}
                   F'(x_k^\delta)^* ( y^\delta - F(x_k^\delta) ) \, .
\end{equation}
In the sequel we address some previous convergence analysis results:

\noindent {\bf (i)} \
For exact data (i.e., $\delta=0$) convergence is proved in \cite[Theorem~2.2]{Ha97},
provided the operator $F$ satisfies adequate regularity assumptions, and
$\{ \alpha_k \}$ satisfies the "exact" condition
\begin{equation} \label{def:alphak-lm}
\| y^\delta - F(x_k^\delta) - F'(x_k^\delta) h_{k,\alpha_k} \|^2
\ = \
\theta \, \| y^\delta - F(x_k^\delta) \|^2 \, ,
\end{equation}
where $h_{k,\alpha_k} = h_k(\alpha_k)$ is given by \eqref{def:lm}, and $\theta < 1$
is an appropriately chosen constant.%
\footnote{It is well known (cf \cite{Gr84}) that $\alpha_k$ is uniquely defined
by \eqref{def:alphak-lm}.}
In the case of inexact data (i.e., $\delta > 0$), semi-convergence is proven
if the iteration in \eqref{def:lm} is stopped according to the discrepancy principle.
%
%
%
The analysis presented in \cite{Ha97} depends on a nonlinearity assumption
on the operator $F$, namely the {\em strong Tangential Cone Condition} (sTCC)
\cite{KalNeuSch08}.

\noindent {\bf (ii)} \
In \cite{BKL10} a convergence analysis for a Kaczmarz version of the LM method, using
constant sequence $\{ \alpha_k  = \alpha \}$, is presented. The convergence proofs
depend once again on a nonlinearity assumption on the operator $F$, namely the
{\em weak Tangential Cone Condition} (wTCC) \cite{EngHanNeu96, KalNeuSch08, HanNeuSch95}.



\noindent {\bf (iii)} \ The algorithm \texttt{REGINN} is a Newton-like method
for solving nonlinear
inverse problems \cite{Rie99}. This iterative algorithm linearizes the
forward operator around the current iterate and subsequently applies a
regularization technique in order to find an approximate solution to the
linearized system, which in turn is added to the current iterate to provide
an update. If wTCC holds true and the iteration is terminated by the
discrepancy principle, then \texttt{REGINN} renders a regularization method
in the sense of \cite{EngHanNeu96}.
If Tikhonov regularization is used for approximating the solution of the
linearized system, then \texttt{REGINN} becomes a variant of the LM method
with a choice of the Lagrange multipliers performed {\em a posteriori}.
In this case, the resulting method is very similar to the one presented
in \cite{Ha97}, but with the difference that the equality in \eqref{def:alphak-lm}
is replaced by an inequality.

\subsection{Criticism on the available choices of the Lagrange multipliers}
Although the proposed choice of $\{ \alpha_k \}$ in \cite{Ha97} is performed
{\em a posteriori}, there is a severe drawback: the calculation of $\alpha_k$ in
\eqref{def:alphak-lm} cannot be performed explicitly. Moreover, computation of
accurate numerical approximations for $\alpha_k$ is highly expensive.

For larger choices of the discrepancy constant, alternative parameter choice rules
are discussed in \cite{Ha97}, namely $\alpha_k = \alpha$ a positive constant, or
$\alpha_k := \| F'(x_k^\delta) \|^2$. However, the use of large values for
discrepancy principle implies in the computation of small stopping indexes,
meaning that LM iteration is interrupted before it can {\color{black} deliver} the best possible
approximate solution.
On the other hand, the constant choice $\{ \alpha_k = \alpha \}$ also has an
intrinsic disadvantage: although the calculation of $\alpha$ demands no numerical
effort, it does not lead to fast convergence of the sequence $\{ x_k^\delta \}$
(this is observed in the numerical experiments presented in \cite{BKL10}).

{\color{black}
The Newton type method proposed in \cite{Rie99} also chooses the Lagrange multiplier
within a range (see also \cite{winkler}).
However, differently from our criteria \eqref{eq:rb}, this range is
defined by a single inequality \cite[Inequality~(2.6)]{Rie99}.
As a consequence, a regularization method (an inner iteration) is needed for the
accurate computation of each multiplier.
}

{\color{black}
In our method, the computation of $\alpha_k$ requires knowledge about the noise
level $\delta > 0$ and the wTCC constant $\eta \in [0,1)$ (see Algorithm~I).
Other Newton type methods (with {\em a posteriori} choice of $\alpha_k$) also
have this characteristic, e.g., see
\cite[Lemma~3.2]{Rie99} and \cite[proof of Theorems~2.2 and~2.3]{Ha97}.
}

\subsection{Outline of the manuscript}
In Section~\ref{sec:rrLM-method} we state the basic assumptions and introduce the
range-relaxed criteria for choosing the Lagrange multipliers.
The algorithm for the corresponding LM type method is presented,
and we prove some preliminary results, which guarantee that our method is well defined.
In Section \ref{sec:conv-analysys} we present the main convergence analysis results,
namely: convergence for exact data, stability and semiconvergence results.
In Section \ref{sec:numeric} numerical experiments are presented for the EIT problem
in a 2D-domain. We compare the performance of our method with other implementations of
the LM method using classical ({\em a priori}) geometrical choices of the Lagrange
multipliers.
Section~\ref{sec:conclusion} is devoted to final remarks and conclusions.

\section{Range-relaxed Levenberg-Marquardt method} \label{sec:rrLM-method}

In this section we introduce a range-relaxed criteria for choosing
the Lagrange multipliers in the Levenberg-Marquardt (LM) method.
Moreover, we present and discuss an algorithm for the resulting LM
type method, here called the {\em range-relaxed Levenberg-Marquardt}
(rrLM) method.

We begin this section by introducing the main assumptions used in this
manuscript. It is worth mentioning that these assumptions are commonly used in the analysis
of iterative regularization methods for nonlinear ill-posed problems
\cite{EngHanNeu96, KalNeuSch08, Sch93a}.

\subsection{Main assumptions}

Throughout this article we assume that the domain of definition $D(F)$
has nonempty interior, and that the initial guess $x_0 \in X$ satisfies
$B_\rho(x_0) \subset D(F)$ for some $\rho > 0$. Additionally,
\medskip

\noindent (A1) \
The operator $F$ and its Fr\'echet derivative $F'$ are continuous. Moreover,
there exists $C > 0$ such that
\begin{equation} \label{eq:dfbounded}
\| F'(x) \| \, \le \, C \, , \quad \ x \in B_\rho(x_0) \, .
\end{equation}
%

\noindent (A2) \
The wTCC holds at some ball $B_{\rho}(x_0)$, with $0 \leq \eta < 1$ and $\rho > 0$, i.e.,
\begin{equation} \label{eq:w-tcc}
\| F(\bar{x}) - F(x) -  F'(x)( \bar{x} - x ) \|_Y \, \leq \,
  \eta \, \| F(\bar{x}) - F(x) \|_Y \, , \qquad
  \forall \ x, \bar{x} \in B_{\rho}(x_0) \, .
\end{equation}

\noindent (A3) \
There exists $x^\star \in B_{\rho/2}(x_0)$ such that $F(x^\star) = y$,
where $y \in Rg(F)$ are the exact data satisfying \eqref{eq:noisy-d}, i.e., $x^\star$
is an arbitrary solution (non necessarily unique).

\subsection{A Levenberg-Marquadt type algorithm}

%
In what follows we introduce an iterative method, which derives from the choice of
Lagrange multipliers proposed in this manuscript (see Step~[3.1] of the Algorithm~I).
\medskip

\begin{normalsize}
\noindent \fbox{
\begin{minipage}[h]{0.97 \textwidth}
{\bf Algorithm~I: Range-relaxed Levenberg-Marquadt}.
\begin{itemize}
\item [[0\!\!\!] ]  Choose an initial guess $x_0 \in X$; \ Set \ $k = 0$.
\item[[1\!\!\!] ] Choose the positive constants $\tau$, $\varepsilon$ and $p$ such that \\[-2ex]
  \begin{equation} \label{def:tau-sigma}
    \tau > \dfrac{1+\eta}{1-\eta} \, , \quad\quad
    0 < \varepsilon < \dfrac{\tau(1-\eta)-(1+\eta)}{\eta\tau} \, , \qquad\quad
    0 < p < 1.
  \end{equation} \\[-6ex]
\item [[2\!\!\!] ]  If \ $\norm{F(x_0) - y^\delta} \leq \tau\delta$, \ then \ $k^* = 0$; \ Stop!
\item [[3\!\!\!] ]  For \ $k \geq 0$ \ do \\[-4ex]
  \begin{itemize}
  \item [ [3.1\!\!\!] ] Compute $\alpha_k > 0$ and  $h_k \in X$, such that \\[-4ex]
    \begin{align}
      & h_k = \big( F'(x_k^\delta)^* F'(x_k^\delta) + \alpha_k I \big)^{-1}
             F'(x_k^\delta)^* (y^\delta - F(x_k^\delta)) \label{eq:hk} \\
      &\norm{y^\delta - F(x_k^\delta) - F'(x_k^\delta)h_k} \, \in \, [c_k , d_k] \label{eq:rb}
    \end{align}  \\[-5ex]
    where  \\[-5.5ex]
    \begin{align}
     & c_k = (1+\varepsilon) \eta \norm{F(x_k^\delta)-y^\delta}
             + (1+\eta) \delta \label{eq:ck} \\
     & d_k = p \, c_k + (1-p) \, \norm{F(x_k^\delta)-y^\delta} \, . \label{eq:dk}
%
    \end{align} \\[-8ex]
  \item [[3.2\!\!\!] ] Set \\[-4ex]
    \begin{equation} \label{eq:xkp1}
      x_{k+1}^\delta \, = \, x_k^\delta + h_k \, .
    \end{equation} \\[-6ex]
%
%
  \item [[3.3\!\!\!] ]  If \ $\norm{F(x_k^\delta) - y^\delta} \leq \tau\delta$,
        \ then \ $k^* = k$; \ Stop! \\
        \mbox{}\,\, Else \ $k = k+1$; \ Go to Step~[3].
  \end{itemize}
\end{itemize}
\end{minipage}
}
\end{normalsize}

\begin{remark}
Due to (A2) and \eqref{def:tau-sigma}, it follows $\tau > 1$. Moreover,
$\big[ \tau(1-\eta)-(1+\eta) \big] (\eta\tau)^{-1} > 0$. Consequently,
the interval used to define $\varepsilon$ in \eqref{def:tau-sigma} is
non-degenerate.
\end{remark}

\begin{remark}
For linear operators $F: X \to Y$, Assumption (A2) is trivially satisfied
with $\eta = 0$. Thus, \,$c_k = \delta$,
\,$d_k = p \delta + (1-p) \norm{F(x_k^\delta) - y^\delta}$ \,and
\eqref{eq:rb} reduces to
$$
\norm{F\, x_k^\delta - y^\delta + F\, h_k^\delta} \, = \,
\norm{F\, x_{k+1}^\delta - y^\delta} \, \in \, [c_k , d_k] \, .
$$
Consequently, the rrLM method in Algorithm~I generalizes the
{\em range-relaxed nonstationary iterated Tikhonov} (rrNIT) method for
linear ill-posed operator equations proposed in \cite{BLS19}.
\end{remark}


{\color{black} From now on we assume that $F'(x) \not= 0$ for $x \in B_\rho(x_0)$.
Notice that this fact follows from Assumption (A2) provided $F$ is
non-constant in $B_\rho(x_0)$.}

The remaining of this section is devoted to verify that, under assumptions
(A1), (A2) and (A3), Algorithm~I is well defined (see Theorem~\ref{th:wd}).
We open the discussion with Lemma~\ref{lm:a0}, where a collection of
preliminary results in Functional and Convex analysis is presented.

\begin{lemma}
  \label{lm:a0}
  {\color{black}
  Suppose $A:X\to Y$ ($A \not= 0$) is a continuous linear mapping, $\bar z\in X$, $b\in Y$
  has a non-zero projection onto the closure of the range of $A$}
  and define, for $\alpha>0$,
  \begin{align}
    \label{eq:za}
    z_{\alpha}&=\argmin_{z \in X}\norm{A(z-\bar z)-b}^2+\alpha\norm{z-\bar z}^2 .
  \end{align}
  The following assertions hold
  \begin{enumerate}
  \item $z_\alpha = \bar z+(A^*A+\alpha I)^{-1} A^* b$;
  \item $\alpha \mapsto \norm{A(z_\alpha-\bar z) - b}$
        is a continuous, strictly increasing function on $\alpha>0$;
  \item $\lim\limits_{\alpha \to 0} \norm{A(z_\alpha-\bar z)-b}
        =\inf\limits_{z \in X}\norm{A(z-\bar z)-b}$;
  \item $\lim\limits_{\alpha \to \infty} \norm{A(z_\alpha-\bar z) - b} = \norm{b}$;
  \item $\norm{A(z_\alpha-\bar z)} \geq \norm{b} - \norm{A(z_\alpha-\bar z)-b} \geq 0$;
  \item $\alpha \leq \norm{A^*b}^2 \, \big[ \norm{b}(\norm{b}-\norm{A(z_\alpha-\bar z)-b}) \big]^{-1}$;
  \item For $z \in X$ and $\alpha > 0$
  \begin{multline}
    \label{eq:pgain}
    \norm{z-\bar z}^2 - \norm{z-z_\alpha}^2  =  \norm{z_\alpha-\bar z}^2 +
    \frac 1 \alpha \big[ \norm{A(z_\alpha-\bar z)-b}^2 - \norm{A(z-\bar z)-b}^2 \big]
    + \frac 1 \alpha \norm{A(z - z_\alpha)}^2 ;
  \end{multline}
  \item For $z\in X$, $z \neq \bar z$, and $\alpha > 0$
  \begin{align}
    \label{eq:alpha_est}
    \alpha \geq \dfrac{\norm{A(z_\alpha-\bar z)-b}^2 - \norm{A(z-\bar z)-b}^2}{\norm{z-\bar z}^2} .
  \end{align}
  \end{enumerate}
\end{lemma}
\begin{proof}
The proofs of items {\em 1.} and {\em 5.} are straightforward.
For a proof of items {\em 2.} to {\em 4.} we refer the reader to \cite{Gr84}.
The proofs of items {\em 6.} and {\em 7.} are adaptations of proofs presented
in \cite{BLS19}, and item {\em 8.} follows from item {\em 7.}
\end{proof}

The next Lemma provides an auxiliary estimate, which is used in the proof of
Proposition~\ref{pr:00}. This proposition is fundamental for establishing
that, as long as the discrepancy is not reached (see Step~[3.3] of Algorithm~I),
two key facts hold true:
(i) it is possible to find a pair $(\alpha_k \in \R^+, \ h_k \in X)$ solving
\eqref{eq:hk}, \eqref{eq:rb} in Step~[3.1] of Algorithm~I; \
(ii) for any sequence $\{ x_k^\delta \}$ generated by Algorithm~I, the
{\em iteration error} $\norm{x^\star - x_k^\delta}$ is monotonically decreasing
in $k$.

\begin{lemma} \label{lm:00}
Let Assumptions (A2) and (A3) hold. Then, for $x^\star$ as in (A3) it holds
$$
\norm{F(x) - y^\delta + F'(x)(x^\star - x)} \, \leq \,
\eta \norm{F(x) - y^\delta} + (1+\eta) \delta \, ,
\ \forall x \in B_\rho(x_0) \, .
$$
\end{lemma}
\begin{proof}
Since $x$, $x^\star \in B_\rho(x_0)$, it follows from (A2) that
\begin{align*}
\norm{F(x)-y^\delta + F'(x)(x^\star - x)}
& = \,    \norm{F(x) - F(x^\star) + F'(x)(x^\star - x) + F(x^\star) - y^\delta}
\\
& \leq \, \eta\norm{F(x) - F(x^\star)} + \norm{F(x^\star) - y^\delta}
\\
& \leq \, \eta \big(\norm{F(x) - y^\delta} + \norm{y^\delta - F(x^\star)} \big)
          + \norm{F(x^\star) - y^\delta} \, .
  \end{align*}
The conclusion follows from this inequality, (A3) and \eqref{eq:noisy-d}.
\end{proof}

\begin{proposition} \label{pr:00}
  Let Assumptions (A2) and (A3) hold. Given $x \in B_\rho(x_0)$, define
  \begin{align}
    (0, +\infty) \ni \alpha \ \mapsto \ 
    \xi_{\alpha} & := \argmin_{\xi\in X}
    \norm{F(x)-y^\delta+F'(x)(\xi-x)}^2 + \alpha \norm{\xi-x}^2 \in X \, .
  \end{align}
  1. \ For every $\alpha > 0$ it holds
  \begin{align} \label{eq:xi_alpha}
  \norm{F'(x)} \, \norm{\xi_\alpha - x} & \geq \norm{F(x) - y^\delta} -
  \norm{F(x) - y^\delta + F'(x)(\xi_\alpha - x)} \, .
  \end{align}
  Additionally, if \ $\norm{F(x) - y^\delta} > \tau \delta$, define the scalars
  \begin{align*}
    c & \, := \, (1+\varepsilon) \eta \norm{F(x) - y^\delta} + (1+\eta) \delta \, ,
    \\
    d & \, := \, p \Big[ (1+\varepsilon) \eta \norm{F(x) - y^\delta} + (1+\eta) \delta \Big]
                 + (1-p) \norm{F(x) - y^\delta} \, ,
  \end{align*}
  and the set \ $J$ $:=$ $\{ \alpha > 0$ $:$ $\norm{F(x) - y^\delta + F'(x)(\xi_\alpha - x)}
  \in [c , d] \}$. Then \\
  2. $J$ is a non-empty, non-degenerate interval; \\
  3. For $\alpha \in J$ and $x^\star$ as in (A3) it holds
  \begin{align}
    \norm{x^\star - x}^2 - \norm{x^\star - \xi_\alpha}^2 \, \geq \, \norm{\xi_\alpha - x}^2 .
  \end{align}
\end{proposition}

\begin{proof}
We adopt the notation: \,$z = x^\star$, \,$z_\alpha = \xi_\alpha$, \,$\bar z = x$,
\,$b = y^\delta - F(x)$ \,and \,$A = F'(x)$.

\noindent
{\bf Add~1.:} Equation \eqref{eq:xi_alpha} follows from Lemma~\ref{lm:a0} (item~5.).

\noindent
{\bf Add~2.:} From the definition of $\varepsilon$ and $\tau$ in \eqref{def:tau-sigma}
it follows that
$$
c \ < \ \big[ \eta\tau + \tau(1-\eta) - (1+\eta) \big] \, \tau^{-1} \norm{F(x)-y^\delta}
    + (1+\eta) \delta \ \leq \ \norm{F(x)-y^\delta}
$$
(the last inequality follows from $\delta \leq \tau^{-1} \norm{F(x) - y^\delta}$).
Since $d$ is a proper convex combination of $c$ and $\norm{F(x)-y^\delta}$, we have
\begin{align} \label{eq:d_estim}
c \, < \, d \, < \, \norm{F(x) - y^\delta} \, .
\end{align}
On the other hand, it follows from Lemma~\ref{lm:00} that
\begin{align} \label{eq:c_estim}
\norm{F(x) - y^\delta + F'(x)(x^\star - x)} \, \leq \,
\eta \norm{F(x) - y^\delta} + (1+\eta) \delta \, < \, c \, .
\end{align}
From \eqref{eq:d_estim}, \eqref{eq:c_estim} it follows that
$$
\inf_z \norm{F(x) - y^\delta + F'(x)(z - x)} \, < \, 
     c \, < \, d \, < \, \norm{F(x)-y^\delta} \, .
$$
Assertion~2. {\color{black} follows} from this inequality and Lemma~\ref{lm:a0} (items~2., 3. and 4.).

\noindent
{\bf Add~3.:} From \eqref{eq:c_estim} and the assumption $\alpha \in J$, we conclude that
$$
\norm{F(x) - y^\delta + F'(x)(x^\star - x)} \, < \, c \, \leq \,
\norm{F(x) - y^\delta + F'(x)(\xi_\alpha - x)} \, .
$$
Assertion~3. {\color{black} follows} from this inequality and Lemma~\ref{lm:a0} (item~7.).
\end{proof}

We are now ready to state and prove the main result of this section.

\begin{theorem} \label{th:wd}
Let Assumptions (A1), (A2) and (A3) hold. Then, Algorithm~I is well defined,
i.e., for $k < k^*$ (the stopping index defined in Step~[3.3]) there
exists a pair $(\alpha_k \in \R^+, \ h_k \in X)$ solving
\eqref{eq:hk}, \eqref{eq:rb}.
Moreover, $k^*$ is finite and any sequence $\set{x_k^\delta}$ generated
by this algorithm satisfies
\begin{align} \label{eq:kgain}
 \norm{x^\star - x_k^\delta}^2 - \norm{x^\star - x_{k+1}^\delta}^2
 \, \geq \,
 \norm{x_k^\delta - x_{k+1}^\delta}^2 \, , \ 0 \leq k < k^* \, .
\end{align}
\end{theorem}
\begin{proof}
Let step $k = 0$ of Algorithm I be its initialization.
We may assume $\norm{F(x_0) - y^\delta} > \tau\delta$
(otherwise the algorithm stops with $k^* = 0$, and the
theorem is trivial).

We use induction for proving this result.
For $k=0$, it follows from Proposition~\ref{pr:00}~(item~2.)
with $x = x_0$, the existence of $(\alpha_0 \in \R^+, \, h_0 \in X)$
solving \eqref{eq:hk}, \eqref{eq:rb}.
Moreover, it follows from Proposition~\ref{pr:00}~(item~3.)
with $x = x_0$, that \eqref{eq:kgain} holds for $k=0$.

Assume by induction that Algorithm~I is well defined up to
step $k_0 > 0$, and that \eqref{eq:kgain} holds for $k = 0,
\dots k_0-1$.
There are two possible scenarios to consider: \\
$\bullet$ \ Case~I: $\norm{F(x_{k_0}^\delta) - y^\delta} \leq \tau\delta$.\\
In this case, the algorithm terminates at iteration $k^* = k_0 \geq 1$,
concluding the proof. \\
$\bullet$ \ Case~II: $\norm{F(x_{k_0}^\delta) - y^\delta} > \tau\delta$.\\
Due to the inductive assumption, $\norm{x^\star - x_{k_0}^\delta} \leq
\norm{x^\star - x_{k_0-1}^\delta} \leq \dots \leq \norm{x^\star - x_0^\delta}$.
From (A3) follows
$$
 \norm{x_{k_0}^\delta - x_0^\delta}
 \, \leq \,
 \norm{x_{k_0}^\delta - x^\star} + \norm{x^\star - x_0^\delta}
 \, \leq \,
 2 \norm{x^\star - x_0^\delta} < \rho \, .
$$
Hence, $x_{k_0}^\delta \in B_\rho(x_0)$. Proposition~\ref{pr:00}~(item~2.)
with $x = x_{k_0}^\delta$, guarantees the existence of a pair
$(\alpha_{k_0} \in \R^+, \, h_{k_0} \in X)$ solving \eqref{eq:hk},
\eqref{eq:rb} as well as the existence of $x_{k_0+1}^\delta \in X$.
The validity of \eqref{eq:kgain} for $k = k_0$ follows from
Proposition~\ref{pr:00}~(item~3.) with $x = x_{k_0}^\delta$.
\medskip

In order to verify the finiteness of the stopping index $k^*$,
notice that, from Proposition~\ref{pr:00}~(item~1.) with \
$x = x_k^\delta$, \ $\alpha = \alpha_k$ \ and \ $\xi_\alpha = x_{k+1}^\delta$,
it follows
$$
  \norm{F'(x_k^\delta)} \, \norm{x_{k+1}^\delta - x_k^\delta} \, \geq \,
  \norm{F(x_k^\delta)-y^\delta} -
  \norm{F(x_k^\delta) - y^\delta + F'(x_k^\delta)(x_{k+1}^\delta - x_k^\delta)} ,
  \ k = 0, \dots k^* - 1 .
$$
From this inequality and the definition of $c_k$ and $d_k$ in Step~[3.1],
it follows that
\begin{eqnarray*}
  \norm{F'(x_k^\delta)} \, \norm{x_{k+1}^\delta - x_k^\delta}
  & \! \geq \! & \norm{F(x_k^\delta)-y^\delta} - d_k
  \      =         \ p \big[ \norm{F(x_k^\delta)-y^\delta} - c_k \big] \\
  & \! = \!    & p \big[ (1-(1+\varepsilon)\eta) \, \norm{F(x_k^\delta)-y^\delta} - (1+\eta) \delta \big] ,
  \ k = 0, \dots k^* - 1 .
\end{eqnarray*}
Since $\norm{F(x_k^\delta)-y^\delta} > \tau \delta$, $0 \leq k < k^*$
and $\varepsilon < \frac 1 \eta - 1$ (see \eqref{def:tau-sigma}), we obtain
from the last inequality
\begin{eqnarray*}
  \norm{F'(x_k^\delta)} \, \norm{x_{k+1}^\delta - x_k^\delta}
  \ \geq \ p \big[ (1-(1+\varepsilon)\eta) \, \tau - (1+\eta) \big] \, \delta
  \   =  \ p\delta \, \eta\tau \Big[ \dfrac{\tau(1-\eta)-(1+\eta)}{\eta\tau} - \varepsilon \Big] ,
\end{eqnarray*}
for $k = 0, \dots k^* - 1$.%
Now, Assumption (A1) implies
\begin{equation} \label{eq:step-estimate}
  \norm{x_{k+1}^\delta - x_k^\delta}
  \ \geq \ \dfrac{p\delta \, \eta\tau}{C}
           \Big[ \dfrac{\tau(1-\eta)-(1+\eta)}{\eta\tau} - \varepsilon \Big]
  \ =: \ \Psi > 0 \, , \ k = 0, \dots k^* - 1 .
\end{equation}
Adding up inequality \eqref{eq:kgain} for $k = 0, \dots\ k^*-1$ and using
\eqref{eq:step-estimate} we finally obtain
$$
 \norm{x^\star - x_0}^2 \, > \,
 \norm{x^\star - x_0}^2 - \norm{x^\star - x_{k^*}^\delta}^2 \, > \,
 \textstyle\sum_{k=0}^{k^*-1} \norm{{\color{black} x_k^\delta - x_{k+1}^\delta}}^2 \, > \,
 k^* \Psi^2 \, ,
$$
from where the finiteness of the stopping index $k^*$ follows.
\end{proof}

\begin{remark}
Assumption (A1) is used only once in the proof of Theorem~\ref{th:wd},
namely in the derivation of \eqref{eq:step-estimate}, which is used to
prove finiteness of the stopping index $k^*$.
\end{remark}


\begin{corollary} \label{cor:reas_wanderer}
Let Assumptions (A1), (A2) and (A3) hold, and assume the data is
exact, i.e., $\delta = 0$. Then, any sequence $\set{x_k}$ generated
by Algorithm~I satisfies
\begin{align} \label{eq:reas_wanderer}
 \textstyle\sum_{k=0}^\infty \ \norm{x_k - x_{k+1}}^2 \, < \, \infty .
\end{align}
\end{corollary}
\begin{proof}
Adding up inequality \eqref{eq:kgain}, we obtain
$$
 \norm{x^\star - x_0}^2 - \norm{x^\star - x_{n+1}}^2 \, > \,
 \textstyle\sum_{k=0}^n \norm{x_k - x_{k+1}}^2 , \ \forall n > 0
$$
and the assertion follows.
\end{proof}


We conclude this section obtaining an estimate for the Lagrange
multipliers $\{ \alpha_k \}$ defined in Step~[3.1] of Algorithm~I.

\begin{proposition} \label{pr:alphaK_est}
Let Assumptions (A2) and (A3) hold. Then the Lagrange multipliers
$\{ \alpha_k \}$ in Algorithm~I satisfy
\begin{align} \label{eq:alphaK_est}
 \alpha_k \geq \rho^{-2} \varepsilon \eta \, \norm{F(x_k^\delta) - y^\delta} \,
 \big[ (1+\varepsilon) \eta \norm{F(x_k^\delta) - y^\delta} + (1+\eta) \delta \big] .
\end{align}
\end{proposition}
\begin{proof}
Take $\alpha = \alpha_k$, \ $z_\alpha = x_{k+1}^\delta$, \ $\bar z = x_k^\delta$,
\ $z = x^\star$, \ $b = y^\delta - F(x_k^\delta)$ \ and \ $A = F'(x_k^\delta)$.
Arguing as in the proof of Lemma~\ref{lm:00} we obtain
\begin{equation} \label{eq:alpha_aux1}
\norm{A(z - \bar z) - b} \, \leq \, \eta \norm{b} + (1+\eta) \delta \, .
\end{equation}
On the other hand, it follows from Step~[3.1] that
\begin{equation} \label{eq:alpha_aux2}
\norm{A(z_\alpha - \bar z) - b} \, \geq \,
(1+\varepsilon) \eta \norm{b} + (1+\eta) \delta \, .
\end{equation}
From \eqref{eq:alpha_aux1} and \eqref{eq:alpha_aux2} we obtain
$\norm{A(z_\alpha - \bar z) - b} - \norm{A(z - \bar z) - b} \geq \varepsilon \eta \norm{b}$.
This last inequality together with \eqref{eq:alpha_est} allow us to estimate
\begin{align*}
\alpha_k \geq & \ \rho^{-2}
                  \big[ \norm{A(z_\alpha - \bar z) - b}^2 - \norm{A(z - \bar z) - b}^2 \big] \\
         \geq & \ \rho^{-2} \big[ \norm{A(z_\alpha - \bar z) - b} + \norm{A(z - \bar z) - b} \big]
                \, \varepsilon \eta \, \norm{b} \\
         \geq & \ \rho^{-2} \varepsilon \eta \, \norm{b} \, \norm{A(z_\alpha - \bar z) - b} \, .
\end{align*}
Estimate \eqref{eq:alphaK_est} {\color{black} follows} from this inequality together with \eqref{eq:alpha_aux2}.
\end{proof}

\section{Convergence analysis} \label{sec:conv-analysys}

We open this section obtaining an estimate, which is similar
in spirit to Lemma~\ref{lm:a0}~(item~7.).

\begin{lemma}
Let Assumptions (A2) and (A3) hold. Then, for $x^\star$ as in (A3) it holds
\begin{multline} \label{eq:gain1}
\norm{x^\star - x_k^\delta}^{2} - \norm{x^\star - x_{k+1}^\delta}^{2} \, \geq \\
  \norm{x_k^\delta - x_{k+1}^\delta}^{2} +
  2 \varepsilon\eta \alpha_k^{-1}
  \norm{F'(x_k) (x_{k+1}^\delta - x_k^\delta) + F(x_k^\delta) - y^\delta} \,
  \norm{F(x_k^\delta) - y^\delta} \, ,
\end{multline}
for $k = 0, \dots , k^* - 1$.
\end{lemma}

\begin{proof}
The polarization identity yields
\begin{equation} \label{eq:gain_00}
\norm{x^\star - x_k^\delta}^{2} - \norm{x^\star - x_{k+1}^\delta}^{2}
 = \norm{x_k^\delta - x_{k+1}^\delta}^{2}
 - 2 \big\langle x_{k+1}^\delta - x_k^\delta , x_{k+1}^\delta - x^\star \big\rangle \, .
\end{equation}
Adopting the notation \,$A := F'({\color{black} x_k^\delta})$, \,$b := y^\delta - F(x_k^\delta)$,
it follows from \eqref{eq:hk} and \eqref{eq:xkp1}
\begin{align}
- \big\langle x_{k+1}^\delta - x_k^\delta , \ x_{k+1}^\delta - x^\star \big\rangle
  &  = \alpha_k^{-1}
       \big\langle A^* (A h_k - b )  , \ x_{k+1}^\delta - x^\star \big\rangle
       \nonumber \\
  &  = \alpha_k^{-1}
       \big\langle A h_k - b , \ A [h_k - (x^\star - x_k^\delta) ] \big\rangle
       \nonumber \\
  &  = \alpha_k^{-1}
       \Big[  \big\langle A h_k - b , A h_k - b \big\rangle
       - \big\langle A h_k - b , A (x^\star - x_k^\delta) - b \big\rangle \Big]
       \nonumber \\
  & \geq \alpha_k^{-1}
       \Big[  \norm{A h_k - b}^{2} - \norm{A h_k - b} \,
       \norm{A (x^\star - x_k^\delta) - b} \Big] \nonumber \\
  &  = \alpha_k^{-1} \norm{A h_k - b}
       \Big[  \norm{A h_k - b} - \norm{A (x^\star - x_k^\delta) - b} \Big] \, .
       \label{eq:gain_01}
\end{align}
However, from Lemma~\ref{lm:00} (with $x = x_k^\delta$) and Algorithm~I
(see \eqref{eq:ck} and \eqref{eq:rb}), it follows
\begin{equation} \label{eq:gain_02}
\norm{A (x^\star - x_k^\delta) - b}
  \ \leq  \  \eta \, \norm{b} + (1 + \eta) \, \delta
  \   =   \  c_k - \varepsilon\eta \, \norm{b}
  \, \leq \, \norm{A h_k - b} - \varepsilon\eta \, \norm{b} \, . 
\end{equation}
Thus, inequality \eqref{eq:gain1} follows substituting \eqref{eq:gain_01} and
\eqref{eq:gain_02} in \eqref{eq:gain_00}.
\end{proof}

The following results are devoted to the analysis of the residuals
$y^\delta - F(x_k^\delta)$ for a sequence $\set{x_k^\delta}$ generated
by Algorithm~I.
In Proposition~\ref{pr:residual_decay} we estimate the decay rate of the
residuals.
Moreover, in Proposition~\ref{pr:residual_serie} we prove the summability
of the {\color{black} series} of squared residuals.

\begin{proposition} \label{pr:residual_decay}
Let Assumptions (A2) and (A3) hold. Then, for any sequence $\set{x_k^\delta}$
generated by Algorithm~I we have
\begin{equation} \label{eq:residual_est}
\norm{ y^\delta - F(x_{k+1}^\delta)} \, \leq \,
\Lambda \, \norm{ y^\delta - F(x_k^\delta)} \, ,
\end{equation}
for $k = 0, \dots , k^* - 1$.
Here $\Lambda := (C_1 + \eta)(1 - \eta)^{-1}$,
\,$C_1 := p (C_0 - 1)  + 1$ \,and
\,$C_0 := (1+\varepsilon) \eta + (1+\eta) \tau^{-1} < 1$.
Additionally, if
\begin{equation} \label{eq:eta_extra_assump}
\eta \ < \ \frac{p+\frac{p}{\tau}}{2 + p(1+\varepsilon) - \frac p \tau} \, ,
\end{equation}
then $\Lambda < 1$, from where it follows \,$k^* = O(|\ln\delta| + 1)$.%
\footnote{Here $k^*$ is the stopping index defined in Step~[3.3] of Algorithm~I.}
\end{proposition}

\begin{proof}
From Algorithm~I (see \eqref{eq:dk}) and (A2), it follows
\begin{align}
\norm{y^\delta - F(x_{k+1}^\delta)}
 & \leq \norm{y^\delta - F(x_k^\delta) - F'(x_k^\delta) h_k}
        + \, \norm{F(x_k^\delta) + F'(x_k^\delta) h_k - F(x_{k+1}^\delta)}
        \nonumber \\
 & \leq d_k + \eta \norm{F(x_k^\delta) - F(x_{k+1}^\delta)} \nonumber \\
 & \leq d_k + \eta \big( \norm{y^\delta - F(x_k^\delta)}
        + \norm{y^\delta - F(x_{k+1}^\delta)} \big) \, ,
   \ 0 \leq k < k^* \, .       \label{eq:b-kp1}
\end{align}
On the other hand, Algorithm~I (see \eqref{eq:ck}) implies
$c_k \leq C_0 \norm{y^\delta - F(x_k^\delta)}$, $0 \leq k < k^*$.
Consequently, $d_k \leq C_1 \norm{y^\delta - F(x_k^\delta)}$, $0 \leq k < k^*$.
Substituting this inequality in \eqref{eq:b-kp1}, we obtain
the estimate \eqref{eq:residual_est}.

To prove the last assertion, observe that $\Lambda < 1$ iff
\eqref{eq:eta_extra_assump} holds true. Moreover, from Algorithm~I
(see Step~[3.3]) and \eqref{eq:residual_est} follows
\,$\tau \delta \leq \norm{y^\delta - F(x_{k^* - 1}^\delta)}
\leq \Lambda^{k^*-1} \norm{y^\delta - F(x_0)}$.
\,Consequently, \eqref{eq:eta_extra_assump} implies $k^* \leq
(\ln\Lambda)^{-1} \ln\big( \tau\delta / \norm{y^\delta - F(x_0)} \big) +1$,
completing the proof.
\end{proof}

\begin{remark}
Inequality \eqref{eq:eta_extra_assump} holds true if $\eta < 1/3$,  $p$ is
sufficiently close to $1$,  $\varepsilon$ is sufficiently close to zero and
$\tau$ is large enough.
Notice that the condition $\eta < 1/3$ is not necessary for the convergence
analysis devised in this manuscript.
\end{remark}

\begin{proposition} \label{pr:residual_serie}
Let Assumptions (A1), (A2) and (A3) hold.
Suppose that no noise is present in the data (i.e., $\delta = 0$).
Then, for any sequence $\set{x_k}$ generated by Algorithm~I
we have
\begin{equation} \label{eq:residual_serie}
\textstyle\sum\limits_{k=0}^\infty \, \norm{y - F(x_k)}^2
\, < \, \infty \, .
\end{equation}
\end{proposition}
\begin{proof}
From Lemma~\ref{lm:a0}~(item~6.) with $\alpha = \alpha_k$, $\bar z = x_k$,
$z_\alpha = x_{x+1}$, $b = y - F(x_k)$ and $A = F'(x_k)$, follows
\begin{eqnarray}
\frac{1}{\alpha_k}
 & \geq & \frac{\norm{y - F(x_k)} \big(\norm{y - F(x_k)} - \norm{F(x_k) - y - F'(x_k) h_k} \big) }
               {\norm{F'(x_k)^* (y - F(x_k))}^2} \nonumber \\
 & \geq & \frac{ {\color{black} \norm{y - F(x_k)} - \norm{F(x_k) - y - F'(x_k) h_k} } }
               {C^2 \norm{y - F(x_k)}} \label{eq:res_ser1}
\end{eqnarray}
(the last inequality follows from (A1)).
Moreover, it follows from Agorithm~I (see \eqref{eq:rb})
$$
\norm{y - F(x_k)} - \norm{F(x_k) - y - F'(x_k) h_k}
\, \geq \, \norm{y - F(x_k)} - d_k 
\, \geq \, p \big( 1 - (1+\varepsilon) \eta \big) \, \norm{y - F(x_k)}
$$
(notice that $(1 - (1+\varepsilon) \eta) >0$ due to \eqref{def:tau-sigma}).
From {\color{black} this} inequality, \eqref{eq:res_ser1} and \eqref{eq:gain1} follows
\begin{eqnarray}
\norm{x^\star - x_0}^2
& \geq & \textstyle\sum\limits_{k=0}^m \, \frac{2\eta}{\alpha_k}
         \norm{F'(x_k) (x_{k+1} - x_k) + F(x_k) - y} \, \norm{F(x_k) - y}
         \nonumber \\
& \geq & \textstyle\frac{2\eta p(1 - (1+\varepsilon) \eta)}{C^2} \sum\limits_{k=0}^m \, 
         \norm{F'(x_k) (x_{k+1} - x_k) + F(x_k) - y} \, \norm{F(x_k) - y} ,
         \label{eq:res_ser2}
\end{eqnarray}
for all $m \in \mathbb{N}$. Finally, \eqref{eq:residual_serie} follows
from \eqref{eq:res_ser2} and the inequality
$\norm{F'(x_k) (x_{k+1} - x_k) + F(x_k) - y} \geq c_k =
(1+\varepsilon)\eta \norm{F(x_k) - y}$ (see Algorithm~I,
\eqref{eq:rb} and \eqref{eq:ck}).
\end{proof}

\begin{remark} \label{rm:residual_serie}
An {\color{black} immediate} consequence of Proposition~\ref{pr:residual_serie} is the fact that
$\norm{F(x_k) - y} \to 0$ as $k \to \infty$.
It is worth noticing that \eqref{eq:res_ser2} and Algorithm~I also
imply the summability of the series
$$
\textstyle\sum\limits_{k=0}^\infty \norm{F'(x_k) (x_{k+1} - x_k) + F(x_k) - y}^2
\mbox{ \ and \ }
\textstyle\sum\limits_{k=0}^\infty \norm{F'(x_k) (x_{k+1} - x_k) + F(x_k) - y}
                                   \norm{F(x_k) - y}
$$
(compare with \cite[inequalities (18a), (18b), (18c)]{BKL10}).
\end{remark}

In the sequel we address the first main result of this section (see
Theorem~\ref{th:noiseless}), namely convergence of Algorithm~I in
the exact data case (i.e., $\delta = 0$). To state this theorem we
need the concept of $x_{0}-${\em minimal-norm solution} of
\eqref{eq:inv-probl}, i.e., the unique $x^{\dagger} \in X$ satisfying
$\norm{x^\dagger-x_0} := \inf\,\set{\norm{x^\ast - x_0} : F(x^\ast) = y$
and $x^\ast \in B_\rho(x_0) }$.

\begin{remark}
Due to (A2), given $x^\ast \in B_{\rho/2}(x_0)$ a solution of
\eqref{eq:inv-probl} and $z \in N( F'(x^\ast) )$, the element
$x^{\ast} + tz \in B_\rho(x_0)$ is also a solution of \eqref{eq:inv-probl}
for all $t \in (-\frac \rho 2 , \frac \rho 2)$.%
\footnote{Indeed, due to (A2) we have
$$
\norm{F(x^\ast + tz) - y} \, = \, \norm{F(x^\ast + tz) - F(x^\ast)}
\, \leq \, \frac{1}{1-\eta} \norm{F'(x^\ast) (x^\ast + tz - x^\ast)}
\, = \, \frac{|t|}{1-\eta} \norm{F'(x^\ast) \, z} = 0 \, .
$$}

Due to (A3), $x^\dagger \in B_{\rho/2}(x_0)$. Thus, the inequality
$\norm{x^\dagger - x_0}^2 \leq \norm{(x^\dagger + tz) - x_0}^2$ holds for all
$t \in (-\frac \rho 2 , \frac \rho 2)$ and all $z \in N(F'(x^{\dagger}))$,
from where we conclude%
\footnote{The conclusion follows from the fact that
$\norm{x^\dagger-x_0}^2 \leq \norm{(x^\dagger+tz) - x_0}^2, \forall
t \in (-\epsilon , \epsilon)$, implies $\langle x^\dagger - x_0 , z \rangle = 0$.}
\begin{equation} \label{eq:x+}
x^\dagger - x_0 \, \in \, N( F'(x^\dagger) )^\perp .
\end{equation}
\end{remark}

\begin{theorem} \label{th:noiseless}
Let Assumptions (A1), (A2) and (A3) hold.
Suppose that no noise is present in the data (i.e., $\delta = 0$).
Then, any sequence $\{ x_k \}$ generated by Algorithm~I either
terminates after finitely many iterations with a solution of
\eqref{eq:inv-probl}, or it converges to a solution of this
equation as $k \to \infty$. Moreover, if
\begin{equation} \label{eq:nucleo}
N(F'(x^\dagger)) \, \subset \, N(F'(x)) \, , \ \forall x \in B_\rho(x_0)
\end{equation}
holds, then $x_k\to x^{\dagger}$ as $k \to \infty$.
\end{theorem}

\begin{proof}
In what follows we adopt the notation \,$A_k := F'(x_k)$,
\,$b_k := y - F(x_k)$.
If for some $k \in \mathbb{N}$, $\norm{y - F(x_k)} = 0$, then
$x_k$ is a solution and Algorithm~I stops with $k^* = k$.
Otherwise, $\{x_k\}_{k\in\mathbb{N}}$ is a Cauchy sequence.
Indeed, fix $m < n$ and choose $\overline{k} \in
\{m, \dots, n\}$ s.t.
\begin{equation} \label{eq:residual2}
\norm{b_{\overline k}} \, \leq \,
\norm{b_k} \ \text{ for all } \ k \in \{m, \dots,n \} \, .
\end{equation}
From the triangle inequality and the polarization identity, it follows
that for any \,$x^\star$ as in (A3)
\begin{align}
\textstyle {\color{black} \frac 1 2} \, \norm{x_n - x_m}^2
& \leq \norm{x_n - x_{\overline k}}^2 + \norm{x_m - x_{\overline k}}^2 \nonumber \\
&  =   \big( \norm{x^\star - x_{n}}^2 - \norm{x^\star - x_{\overline k}}^2
       + 2 \langle x_n - x_{\overline k} \ , x_{\overline k} - x^\star \rangle \big) \nonumber \\
& \quad + \big( \norm{x^\star - x_m}^2 - \norm{x^\star - x_{\overline k}}^2
       + 2 \langle x_m - x_{\overline k} \ , x_{\overline k} - x^\star \rangle \big) \, .
       \label{eq:iner_prods}
\end{align}
Since the sequence $\{ \norm{x^\star - x_n} \}_{n\in\mathbb{N}}$ is
non-negative and non-increasing {\color{black} (see \eqref{eq:kgain})}, it converges. Therefore, the difference
$\norm{x^\star - x_n}^2 - \norm{x^\star - x_{\overline k}}^2$ as well
as $\norm{x^\star - x_m}^2 - \norm{x^\star - x_{\overline k}}^2$ both
converge to zero as $m \to \infty$. It remains to estimate the inner
products in \eqref{eq:iner_prods}. Notice that
\begin{align}
\big| \langle x_n - x_{\overline k} \ , x_{\overline k} - x^\star \rangle
  + \langle x_m - x_{\overline k} \ , x_{\overline k} - x^\star \rangle \big|
& \leq \big| \langle x_n - x_m \ , x_{\overline k} - x^\star \rangle \big|
       \nonumber\\
& \leq \textstyle\sum\limits_{k=m}^{n-1}
       \big| \langle x_{k+1} - x_k \ , x_{\overline{k}} - x^\star \rangle \big|
       \nonumber\\
&  = \textstyle\sum\limits_{k=m}^{n-1} \frac{1}{\alpha_k}
     \big| \langle A_k^* ( A_k h_k - b_k ) \ , x_{\overline k} - x^\star \rangle \big|
     \nonumber\\
& \leq \textstyle\sum\limits_{k=m}^{n-1} \frac{1}{\alpha_k}
       \norm{A_k h_k - b_k} \, \norm{A_k (x_{\overline k} - x^\star)} \, ,
       \label{eq:a0}
\end{align}
with $h_k$ as in \eqref{eq:hk}.
However, from \eqref{eq:residual2} and (A2) follows
\begin{align*}
\norm{A_k (x_{\overline k} - x^\star)}
& \leq \norm{A_k (x_{\overline k} - x_k)} + \norm{ A_k (x^\star - x_k)} \\
& \leq \norm{F(x_{\overline k}) - F(x_k) - A_k (x_{\overline k} - x_k)}
     + \norm{F(x_{\overline k}) - F(x_k)} \\
&    + \norm{F(x^\star) - F(x_k) - A_k (x^\star - x_k)}
     + \norm{F(x^\star) - F(x_k)} \\
& \leq (\eta+1) \, \norm{F(x_{\overline k}) - F(x_k)}
     + (\eta+1) \, \norm{y - F(x_k)} \\
& \leq 2 (\eta+1) \, \norm{y - F(x_k)} + (\eta+1) \norm{y - F(x_{\overline k})} \\
& \leq 3 (\eta+1) \, \norm{y - F(x_k)} \, .
\end{align*}
Substituting this last inequality in \eqref{eq:a0}, and using
\eqref{eq:gain1} (with $x_k^\delta = x_k$, $y^\delta = y$) we obtain
\begin{align*}
\big| \langle x_n - x_{\overline k} \ , x_{\overline k} - x^\star \rangle
     + \langle x_m - x_{\overline k} \ , x_{\overline k} - x^\star \rangle \big|
& \leq 3(\eta+1) \textstyle\sum\limits_{k=m}^{n-1} \frac{1}{\alpha_k}
       \norm{A_k h_k - b_k} \, \norm{b_k} \\
& \leq \frac{3(\eta+1)}{2\varepsilon\eta}
       \textstyle\sum\limits_{k=m}^{n-1}
       \big( \norm{x^\star - x_k}^2 - \norm{x^{\ast}-x_{k+1}}^2 \big) \\
&  = \frac{3(\eta+1)}{2\varepsilon\eta}
     \big[ \norm{x^\star - x_m}^2 - \norm{x^\star - x_n}^2 \big]
     \to 0
\end{align*}
as $m \to \infty$. Thus, it follows from \eqref{eq:iner_prods}
that $\norm{x_n - x_m} \to 0$ as $m \to \infty$, proving that
$\{ x_k \}_{k\in\mathbb{N}}$ is indeed a Cauchy sequence.

Since $X$ is complete, $\{x_k\}$ converges to some $x_\infty \in X$
as $k \to \infty$.
On the other hand, $\norm{y - F(x_k)} \to 0$ as $k \to \infty$
(see Remark~\ref{rm:residual_serie}). Consequently, $x_\infty$
is a solution of \eqref{eq:inv-probl} proving the first assertion.

In order to prove the last assertion notice that,
if \eqref{eq:nucleo} hold, then
$$
x_{k+1} - x_k \, = \, \alpha_k^{-1} A_k^* (A_k h_k - b_k) \, \in \, R(F'(x_k)^*)
              \, \subset \, N(F'(x_k))^\perp
              \, \subset \, N(F'(x^\dagger))^\perp ,
              \ k = 0, 1, \dots ,
$$
from where we conclude that $x_k - x_0 \in N(F'(x^\dagger))^\perp$,
$k \in \mathbb N$.
Since $x^\dagger - x_0 \in N(F'(x^\dagger))^\perp$ (see \eqref{eq:x+}),
it follows that $x_k - x^\dagger \in N(F'(x^\dagger))^\perp,
k \in \mathbb N$. Consequently,
$x_\infty - x^\dagger = \lim_k x_k - x^\dagger \in N(F'(x^\dagger))^\perp$.
However, (A2) implies
$ \norm{F'(x^\dagger)(x_\infty - x^\dagger)} \leq
(1+\eta) \norm{F(x_\infty) - F(x^\dagger)} = 0$,
from what follows $x_{\infty} - x^\dagger \in N( F'(x^\dagger))$.
Thus, $x_{\infty}-x^{\dagger}=0$.
\end{proof}

We conclude this section adressing the last two main results, namely:
Stability (Theorem~\ref{th:stabil}) and Semi-Convergence (Theorem~\ref{th:reg}).
The following definition is quintessential for the discussion of these
results.

\begin{definition} \label{def:sucessor}
A vector $z \in X$ is a \textbf{sucessor} of $x_k^\delta$ if

$\bullet$ $k < k^*$;

$\bullet$ There exists $(\alpha_k >0 , \ h_k \in X)$ satisfying \eqref{eq:hk},
\eqref{eq:rb}, such that $z = x_k^\delta + h_k$;
\end{definition}

Notice that Theorem~\ref{th:noiseless} guarantees that the sequence
$\{ x_k \}_{k\in\mathbb{N}}$ converges to a solution of $F(x) = y$
whenever $x_{k+1}$ is a sucessor of $x_k$ for every $k\in\mathbb{N}$.
In this situation, we call $\{ x_k \}_{k\in\mathbb{N}}$ a
\textit{noiseless sequence}.

\begin{theorem}[Stability] \label{th:stabil}
Let Assumptions (A1), (A2) and (A3) hold, and $\{ \delta_j \}_{j\in\mathbb{N}}$
be a positive zero-sequence.
Assume that the (finite) sequences $\{ x_k^{\delta_j}\}_{0\leq k\leq k^*(\delta_j)}$,
$j \in \mathbb{N}$, are fixed,%
\footnote{Notice that the stopping index $k^*$ in Step~[3.3] depends on
$\delta$, i.e., $k^* = k^*(\delta)$.}
where $x_{k+1}^{\delta_j}$ is a sucessor of $x_k^{\delta_j}$.
Then, there exists a noiseless sequence $\{ x_k \}_{k\in\mathbb{N}}$
such that, for every fixed $k \in \mathbb{N}$, there exists a
subsequence $\{ \delta_{j_m} \}_{m\in\mathbb{N}}$ (depending on $k$)
satisfying
$$
x_\ell^{\delta_{j_m}} \to x_\ell \ \text{ as } \
m \to \infty \, , \qquad \text{ for } \ \ell = 0, \dots, k \, .
$$
\end{theorem}

\begin{proof}
We use an inductive argument. Since $x_0^\delta = x_0$ for every $\delta \geq 0$,
the assertion is clear for $k = 0$.
Our main argument consists of repeatedly choosing a subsequence of the current
subsequence. In order to avoid a notational overload, we denote a subsequence
of $\{ \delta_j \}_j$ again by $\{ \delta_j \}_j$.

Suppose by induction that the assertion holds true for some $k \in \mathbb{N}$,
i.e., that there exists a subsequence $\{ \delta_j \}_j$ and $\{ x_\ell \}_{\ell=0}^k$
satisfying
$$
x_\ell^{\delta_j} \to x_\ell \ \text{ as } \ j \to \infty \, ,
\qquad \text{ for } \ \ell = 0, \dots, k \, ,
$$
where $k < k^*(\delta_j)$ and \,$x_{\ell+1}$ is a sucessor of $x_{\ell}$, for
$\ell = 0, \dots, k-1$.
Since $x_{k+1}^{\delta_j}$ is a sucessor of $x_k^{\delta_j}$ (for each $\delta_j$),
there exists (for each $\delta_j$) a positive number $\alpha_k^{\delta_j}$ such
that $x_{k+1}^{\delta_j} = x_k^{\delta_j} + h_k^{\delta_j}$,
with $h_k^{\delta_j}$ as in \eqref{eq:hk} and
\begin{equation} \label{eq:a1}
\norm{F(x_k^{\delta_j}) - y^{\delta_j} + F'(x_k^{\delta_j}) h_k^{\delta_j}}
\, \in \, [c_k^{\delta_j} , d_k^{\delta_j}] \, .
\end{equation}
Our next goal is to prove the existence of a sucessor $x_{k+1}$ of $x_k$ and of a
subsequence $\{ \delta_j \}_j$ of the current subsequence such that
$x_{k+1}^{\delta_j} \to x_{k+1}$ as $j \to \infty$, ensuring that
\begin{equation} \label{eq:induction}
x_\ell^{\delta_j} \to x_\ell \ \text{ as }\ j \to \infty \, ,
\quad \text{ for } \ \ell = 0 , \dots, k+1 \, ,
\end{equation}
and completing the inductive argument.
We divide this proof in 4 steps as follows:

Step~1. We find a vector $z \in X$ such that, for some subsequence $\{ \delta_j \}_j$
\begin{equation} \label{eq:b1}
h_k^{\delta_j} \rightharpoonup z \ \text{ as }\ j \to \infty \, ,
\end{equation}

Step~2. We define
\begin{equation} \label{eq:b3}
\alpha_k := \underset{j\to\infty}{\lim\inf} \, \alpha_k^{\delta_j}
\end{equation}

and prove that $\alpha_k > 0$, which permit us to define
$h_k$ as in \eqref{eq:hk} as well as $x_{k+1} := x_k + h_k$.

Step~3. We show that $h_k=z$, which ensures that $h_k^{\delta_j} \rightharpoonup h_k$.

Step~4. We validate that
\begin{equation} \label{nconvergence}
\norm{h_k^{\delta_j}}  \to \norm{h_k} \, , \ \text{ as } \ j \to \infty \, ,
\end{equation}

which together with $h_k^{\delta_j} \rightharpoonup h_k$ proves that
$h_k^{\delta_j}\to h_k$ and, consequently, $x_{k+1}^{\delta_j} \to x_{k+1}$.

Finally, we prove that $x_{k+1}$ is a sucessor of $x_k$, which validates
\eqref{eq:induction}.

\noindent Proof of Step~1. Since the sequence $\{ h_k^{\delta_j} \}_{j\in\mathbb{N}}$
is bounded (see \eqref{eq:kgain}), there exists a subsequence $\{ \delta_j \}$ of the current
subsequence, and a vector $z \in X$ such that \eqref{eq:b1} holds. Consequently,
\begin{equation} \label{eq:b2}
A_k^{\delta_j} h_k^{\delta_j} - b_k^{\delta_j} \ \rightharpoonup \
A_k z - b_k \, , \ \text{ as } \ j \to \infty
\end{equation}
(here $A_k^{\delta_j} = F'(x_k^{\delta_j})$, $A_k = F'(x_k)$,
$b_k^{\delta_j} = y^{\delta_j} - F(x_k^{\delta_j})$, $b_k = y - F(x_k)$).

\noindent Proof of Step~2. If $\alpha_k$ in \eqref{eq:b3} is not positive, we
conclude from (A2)
\begin{align*}
\lim\inf_j \norm{ A_k^{\delta_j} h_k^{\delta_j} - b_k^{\delta_j} }^{2}
  & \leq \lim\inf_j \, T_{k,\delta_j,\alpha_k^{\delta_j}} (h_k^{\delta_j})
    \leq \lim\inf_j \, T_{k,\delta_j,\alpha_k^{\delta_j}} (x^\dag - x_k) \\
  &   =  \lim\inf_j \big(  \norm{ A_k^{\delta_j}(x^\dag - x_k) - b_k^{\delta_j} }^{2}
         + \alpha_k^{\delta_j} \norm{ x^\dag - x_k }^{2} \big) \\
  &   =  \norm{ A_k(x^\dag - x_k) - b_k }^{2} \, \leq \, \eta^{2} \norm{ b_k }^{2}
\end{align*}
(here $T_{k,\delta,\alpha}(h) := \norm{F'(x_k^\delta)h - y^\delta + F(x_k^\delta)}^2
+ \alpha \norm{h}^2$). This leads to the contradiction
$$
c_k = \lim_j c_k^{\delta_j} \leq \lim\inf_j \norm{ A_k^{\delta_j} h_k^{\delta_j} - b_k^{\delta_j} }
      \leq \eta \norm{ b_k } < c_k \, .
$$
Thus $\alpha_k > 0$ holds.
We define $T_{k,\alpha}(h) := T_{k,\delta,\alpha}(h)$ with $\delta = 0$,
$h_k := \arg\min_{h \in X}\, T_{k,\alpha_k}(h)$, and $x_{k+1} := x_k + h_k$.
In order to prove that $x_{k+1}$ is a sucessor of $x_k$, it is necessary to prove that
\begin{equation}  \label{eq:a2}
c_k \leq \norm{ A_k h_k - b_k } \leq d_k \, .
\end{equation}
We first prove that $h_k = z$ (see Step~3).

\noindent Proof of Step~3. From \eqref{eq:b1}, \eqref{eq:b2} and \eqref{eq:b3},
it follows
\begin{align*}
T_{k,\alpha_k}(z) &  = \norm{ A_k z - b_k }^2 + \alpha_k \norm{z}^2
\ \leq \ \lim\inf_j (\norm{A_k^{\delta_j} h_k^{\delta_j} - b_k^{\delta_j} }^2
         + \alpha_k^{\delta_j} \norm{ h_k^{\delta_j} }^2)
\\
&  = \lim\inf_j T_{k,\delta_j,\alpha_k^{\delta_j}} (h_k^{\delta_j})
\ \leq\ \lim\inf_j T_{k,\delta_j,\alpha_k^{\delta_j}}(h_k) \ = \ T_{k,\alpha_k}(h_k) \, .
\end{align*}
Since $h_k$ is the unique minimizer of $T_{k,\alpha_k}$, we conclude that
$h_k=z$. Thus, $h_k^{\delta_j} \rightharpoonup h_k$ as $j \to \infty$.
The last inequalities also ensure that
$\lim\inf_j T_{k,\delta_j,\alpha_k^{\delta_j}}(h_k^{\delta_j}) = T_{k,\alpha_k}(h_k)$.
This guarantees the existence of a subsequence satisfying
\begin{equation} \label{eq:b4}
\lim\limits_{j\to \infty} \, T_{k,\delta_j,\alpha_k^{\delta_j}} (h_k^{\delta_j})
\ = \ T_{k,\alpha_k}(h_k) \, .
\end{equation}

\noindent Proof of Step~4. 
The goal is to validate \eqref{nconvergence}, which, together with
$h_k^{\delta_j} \rightharpoonup h_k$, imply $h_k^{\delta_j} \to h_k$.
Consequently, \eqref{eq:a2} follows from \eqref{eq:a1}.
This ensures that $x_{k+1}$ is a sucessor of $x_k$ and validates
\eqref{eq:induction}, completing the proof of the theorem.

We first prove the existence of a constant $\alpha_{\max,k}$ such that
$$
\alpha_k^{\delta_j} \leq \alpha_{\max,k}
\ \text{ for all }\ j \in \mathbb{N} \, .
$$
Indeed, if such a constant did not exist, we could find a subsequence
satisfying $\alpha_k^{\delta_j}\to \infty$ as $j\to \infty.$
Thus, since
$$
\alpha_k^{\delta_j} \norm{ h_k^{\delta_j} }^2
\ \leq \ T_{k,\delta_k,\alpha_k^{\delta_j}} (h_k^{\delta_j})
\ \leq \ T_{k,\delta_k,\alpha_k^{\delta_j}} (0)
\   =  \ \norm{ b_k^{\delta_j} }^2 ,
$$
we would have,
$$
\lim\limits_{j \to \infty} \alpha_k^{\delta_j} \norm{ h_k^{\delta_j} }^2
\ \leq \ \norm{ b_k }^2 \ < \ \infty \, ,
$$
which would imply $h_k^{\delta_j} \to 0$. Consequently,
$$
\lim\limits_{j \to \infty} \norm{ A_k^{\delta_j} h_k^{\delta_j} - b_k^{\delta_j} }
\ = \ \norm{ b_k }
\ > \ d_k \ = \ \lim\limits_{j \to \infty} d_k^{\delta_j} ,
$$
which would imply the contradiction
$\norm{ A_k^{\delta_j} h_k^{\delta_j} - b_k^{\delta_j} } \ > \ d_k^{\delta_j}$,
for $j$ large enough.

Now we validate \eqref{nconvergence}. This proof follows the lines of \cite[Lemma~5.2]{MaRi15}.
Define
\begin{align*}
a_j := \norm{ h_k^{\delta_j} }^2,
\quad
a := \lim\sup a_j \,,
\quad
c := \norm{ h_k } ^{2},
\quad
re_j := \norm{ A_k^{\delta_j} h_k^{\delta_j} - b_k^{\delta_j} }^2,
\quad
re := \lim\inf re_j \, .
\end{align*}
As $\norm{ h_k } \leq \lim\inf \norm{ h_k^{\delta_j} }$, it suffices
to prove that $a \leq c$. Assume the contrary. From \eqref{eq:b4},
there exists a number $N_1 \in \mathbb{N}$ such that
\begin{equation} \label{eq:c1}
j \ \geq \ N_1 \ \Longrightarrow \
           T_{k,\delta_j,\alpha_k^{\delta_j}} (h_k^{\delta_j})
  \  <  \  T_{k,\alpha_k}(h_k) + \alpha_k \frac{a-c}{2} \, .
\end{equation}
From definition of $\lim\inf,$ there exist constants $N_2$,
$N_3 \in \mathbb{N}$ such that
\begin{equation} \label{eq:c2}
j \geq N_2 \Longrightarrow re_j \geq re - \alpha_k (a-c) / 6
\end{equation}
and
\begin{equation} \label{eq:c3}
j \geq N_3 \Longrightarrow \alpha_k^{\delta_j} \geq \alpha_k - \alpha_k (a-c)/6a \, .
\end{equation}
Moreover, from definition of $\lim\sup$, we conclude that for each
$M \in \mathbb{N}$ fixed, there exists an index $j \geq M$ such that
\begin{equation} \label{eq:c4}
a_j \geq a - \alpha_k (a-c) / (6\alpha_{\max,k}) \, .
\end{equation}
Therefore, for $M := \max\{ N_1, N_2, N_3 \}$, there exists an index $j \geq M$ such that
\begin{align*}
T_{k,\alpha_k}(h_k)
& \leq re + \alpha_k c = re + (\alpha_k - \alpha_k^{\delta_j}) a
       + \alpha_k^{\delta_j} (a - a_j) + \alpha_k^{\delta_j} a_j - \alpha_k(a - c)
\\
& \leq \textstyle (re_j + \alpha_k \frac1 6 (a-c)) + \alpha_k \frac1 6 (a-c)
       + \alpha_k \frac1 6 (a-c) + \alpha_k^{\delta_j} a_j - \alpha_k (a-c)
\\
&   =  \textstyle re_j + \alpha_k^{\delta_j} a_j - \alpha_k \frac1 2 (a-c)
    =  T_{k,\delta_j,\alpha_k^{\delta_j}} (h_k^{\delta_j})
       - \alpha_k \frac1 2 (a-c) < T_{k,\alpha_k}(h_k) \, ,
\end{align*}
where the second inequality follows from \eqref{eq:c2},
\eqref{eq:c3}, \eqref{eq:c4}, while the last inequality follows
from \eqref{eq:c1}. This leads to the obvious contradiction
$T_{k,\alpha_k}(h_k) < T_{k,\alpha_k}(h_k)$, proving that
$a \leq c$ as desired. Thus, \eqref{nconvergence} holds and
the proof is complete.
\end{proof}

\begin{theorem}[Regularization] \label{th:reg}
Let Assumptions (A1), (A2) and (A3) hold, and $\{ \delta_j \}_{j\in\mathbb{N}}$
be a positive zero-sequence.
Assume that the (finite) sequences $\{ x_k^{\delta_j} \}_{0 \leq k \leq k^*(\delta_j)}$,
$j \in \mathbb{N}$, are fixed, where $x_{k+1}^{\delta_j}$ is a sucessor
of $x_k^{\delta_j}$.
Then, every subsequence of $\{ x_{k^*(\delta_j)}^{\delta_j} \}_{j\in\mathbb{N}}$
has itself a subsequence converging strongly to a solution of \eqref{eq:inv-probl}.
\end{theorem}
\begin{proof}
Since any subsequence of $\{ \delta_j \}_{j\in\mathbb{N}}$
is itself a positive zero-sequence, it suffices to prove that
$\{ x_{k^*(\delta_j)}^{\delta_j} \}_{j\in\mathbb{N}}$ has a
subsequence converging to a solution. We consider two cases:

\noindent {\bf Case~1.} The sequence $\{ k^*(\delta_j) \} _{j\in\mathbb{N}}$ is bounded. \\
Thus, there exists a constant $M\in\mathbb{N}$ such that
$k^*(\delta_j) \leq M$ for all $j \in \mathbb{N}$. Thus, the sequence
$\{ x_{k^*(\delta_j)}^{\delta_j} \}_{j \in \mathbb{N}}$ splits into at
most $M+1$ subsequences having the form
$\{ x_m^{\delta_{j_{n}}} \}_{n \in \mathbb{N}}$, with fixed $m \leq M$.
Pick one of these subsequences. From Theorem~\ref{th:stabil}, this
subsequence has itself a subsequence (again denoted by
$\{ x_m^{\delta_{j_{n}}} \}_{n \in \mathbb{N}}$) converging to some $x_m \in X$,
i.e.,
$$
\lim_{n\to \infty} x_{k^*(\delta_{j_{n}})}^{\delta_{j_{n}}}
\ =\ \lim_{n \to \infty} x_m^{\delta_{j_{n}}} \ =\ x_m \, .
$$
Notice that $x_m$ is a solution of \eqref{eq:inv-probl}. Indeed,
\begin{align*}
\norm{ y - F(x_{m}) }
 & =    \lim_{n\to\infty} \norm{ y - F( x_{k^*(\delta_{j_n})}^{\delta_{j_n}} ) } \\
 & \leq \lim_{n\to\infty} \big( \norm{ y - y^{\delta_{j_n}} }
        + \norm{ y^{\delta_{j_n}} - F( x_{k^*(\delta_{j_n})}^{\delta_{j_n}} ) } \big) \\
 & \leq \lim_{n\to\infty} (\tau + 1) \, \delta_{j_n} \ = \ 0 \, .
\end{align*}

\noindent {\bf Case~2.} The sequence $\{ k^*(\delta_j) \} _{j\in\mathbb{N}}$ is not bounded. \\
Thus, there is a subsequence such $k^*(\delta_j) \to \infty$ as $j \to \infty$.
Let $\varepsilon > 0$ be given and consider the noiseless sequence $\{ x_k \}_{k\in\mathbb{N}}$
constructed in last theorem. Since $x_{k+1}$ is a sucessor of $x_k$ for all
$k\in\mathbb{N}$, $\{ x_k \}_{k\in\mathbb{N}}$ converges to some solution $x^{\ast}$ of
\eqref{eq:inv-probl} (see Theorem~\ref{th:noiseless}). Then, there exists $M =
M(\varepsilon) \in \mathbb{N}$ such that
$$
\norm{ x_{M} - x^{\ast} } \ < \ \textstyle \frac 1 2 \, \varepsilon \, .
$$
On the other hand, there exists $J \in \mathbb{N}$ such that $k^*(\delta_j) \geq M$,
for $j \geq J$.
Consequently, it follows from the monotonicity of the iteration error (see
Theorem~\ref{th:wd}) that
$$
j \geq J \ \Longrightarrow \ \norm{ x_{k^*(\delta_j) }^{\delta_j} - x^\ast }
  \leq \norm{ x_{M}^{\delta_j} - x^\ast } \, .
$$
Moreover, it follows from Theorem~\ref{th:stabil}, the existence of a subsequence
$\{ \delta_{j_m} \}$ (depending on $M(\varepsilon)$) and the existence of $N \in \mathbb{N}$
such that
$$
m \geq N \ \Longrightarrow \ \norm{ x_M^{\delta_{j_m}} - x_M }
  \ < \ \textstyle\frac 1 2 \varepsilon \, .
$$
Consequently, for $m \geq \max \{ J, N \}$ (which simultaneously guarantees
$j_m \geq m \geq J$ and $m \geq N$), it holds
\begin{equation} \label{eq:a3}
\norm{ x_{k^*(\delta_{j_m})}^{\delta_{j_m}} - x^\ast }
  \leq \norm{ x_M^{\delta_{j_m}} - x^\ast }
  \leq \norm{ x_M^{\delta_{j_m}} - x_M } + \norm{ x_M - x^\ast }
  \ < \ \varepsilon \, .
\end{equation}

Notice that the subsequence $\{ \delta_{j_m} \}$ depends on $\varepsilon$.
We now construct an $\varepsilon$-independent subsequence using a diagonal argument: for
$\varepsilon = 1$ in \eqref{eq:a3}, there is a subsequence of $\{ \delta_j \}$
(called again $\{ \delta_j \}$) and ${j_1} \in \mathbb{N}$ such that
$$
\norm{ x_{k^*(\delta_{j_1})}^{\delta_{j_1}} - x^{\ast} } \ < \ 1 \, .
$$
Now, for $\varepsilon = 1/2$, there exists a subsequence $\{ \delta_j \}$
of the previous one, and $j_2 > j_1$ such that
$$
\norm{ x_{k^*(\delta_{j_2})}^{\delta_{j_2}} - x^\ast} \ < \ 2^{-1} \, .
$$
Arguing in this way, we construct a subsequence $\{ \delta_{j_n} \}_{n\in\mathbb{N}}$
satisfying
$$
\norm{ x_{k^*(\delta_{j_n})}^{\delta_{j_n}} - x^{\ast} } \ < \ n^{-1} ,
$$
from what follows $\lim x_{k^*(\delta_{j_n})}^{\delta_{j_n}} = x^\ast$.
\end{proof}

\begin{remark}
If the solution of \eqref{eq:inv-probl} referred in Theorem~\ref{th:stabil}
were independent of the chosen subsequence, then any subsequence of
$\{ x_{k^*(\delta_j)}^{\delta_j} \}_{j\in\mathbb{N}}$ would have itself
a subsequence converging to the same solution.
This would be enough to ensure that the whole sequence
$\{ x_{k^*(\delta_j)}^{\delta_j} \}_{j\in\mathbb{N}}$ converges to $x^\ast$.

However, $x^\ast$ in the above proof depends on the noiseless
sequence $\{ x_n \}_{n\in\mathbb{N}}$ (whose existence is guaranteed by
Theorem~\ref{th:stabil}), which in turn depends on the fixed sequences
$\{ x_k^{\delta_j} \}_{0 \leq k \leq k^*(\delta_j)}$, $j \in \mathbb{N}$.
Consequently, if different subsequences of $\{ \delta_j \}_{j\in\mathbb{N}}$
are chosen, the solution of \eqref{eq:inv-probl} refered in Theorem~\ref{th:stabil}
can be different.
\end{remark}

\begin{corollary} \label{cor:reg}
Under the assumptions of Theorem~\ref{th:stabil}, the following assertions
hold true:
\begin{enumerate}
\item The sequence $\{  x_{k^*(\delta_j)}^{\delta_j} \}_{j\in\mathbb{N}}$
splits into convergent subsequences, each one converges to a solution of
\eqref{eq:inv-probl};

\item If $x^\star$ in (A3) is the unique solution of \eqref{eq:inv-probl} in
$B_\rho(x_0)$, then $x_{k^*(\delta_j)}^{\delta_j} \to x^\star$ as $j\to \infty$;

\item If the null-space condition \eqref{eq:nucleo} holds, then
$\{ x_{k^*(\delta_j)}^{\delta_j} \}$ converges to the $x_0$-minimal-norm solution
$x^{\dagger}$ as $j\to \infty$.
\end{enumerate}
\end{corollary}
\begin{proof}
The proof of Assertion~1. is straightforward.
Assertion~2. follows from the fact that, if $x^\star$ is the
unique solution of \eqref{eq:inv-probl} in $B_\rho(x_0)$, then
any subsequence of $\{ x_{k^*(\delta_j)}^{\delta_j} \}_{j\in\mathbb{N}}$
has itself a subsequence converging to $x^\star$.
To prove Assertion~3. notice that, if \eqref{eq:nucleo} holds then any
noiseless sequence $\{ x_n \}_{n\in\mathbb{N}}$ converges to $x^{\dagger}$
(Theorem~\ref{th:noiseless}). Thus, any subsequence of
$\{ x_{k^*(\delta_j)}^{\delta_j} \}_{j\in\mathbb{N}}$ has itself a
subsequence converging to $x^{\dagger}$, and the proof follows.
\end{proof}

\section{Numerical experiments} \label{sec:numeric}

\subsection{The model problem and its discretization} \label{ssec:4.1}

We test the performance of our method applying it to the non-linear and
ill-posed inverse problem of EIT (Electrical Impedance Tomography)
introduced by Calder\'{o}n~\cite{calderon}. A survey article concerning this
problem is \cite{borcea}.

Let $\Omega \subset \mathbb{R}^2$ be a bounded and simply connected
Lipschitz domain. The EIT problem consists in applying different
configurations of electric currents on the boundary of $\Omega $ and then
reading the resulting voltages on the boundary of $\Omega $ as well. The
objective is recovering the electric conductivity in the whole of set $%
\Omega .$ This problem is governing by the variational equation
\begin{equation}  \label{eit}
\int_{\Omega} \gamma \nabla u \nabla \varphi = \int_{\partial \Omega} g
\varphi \ \text{ for all } \varphi \in H_{\diamondsuit}^1 ( \Omega ) \, ,
\end{equation}
where $g\colon \partial \Omega \rightarrow \mathbb{R}$ represents the
electric current, $\gamma \colon \Omega \rightarrow \mathbb{R}$ is the
electric conductivity and $u\colon \Omega \rightarrow \mathbb{R}$ represents
the electric potential. Employing the Lax-Milgram Lemma, one can prove that,
for each $g \in L_{\diamondsuit}^2 ( \partial \Omega) := \{ v \in L^2 (
\partial \Omega ) : \int_{\partial\Omega} v = 0 \}$ and $\gamma \in
L_{+}^{\infty} (\Omega) := \{ v \in L^\infty(\Omega) : v \geq c > 0 \text{ a.e.
in } \Omega \}$ fixed, there exists a unique $u \in
H_{\diamondsuit}^1(\Omega) := \{ v \in H^1(\Omega) : \int_{\partial\Omega} v
= 0 \}$ satisfying \eqref{eit}. The voltage $f\colon \partial \Omega
\rightarrow \mathbb{R}$ is the trace of the potencial $u$ ($f = u
\vert_{\partial\Omega}$), which belongs to $L_{\diamondsuit}^2(\partial%
\Omega)$.

For a fixed conductivity $\gamma \in L_{+}^\infty(\Omega)$, the bounded
linear operator $\Lambda_{\gamma}\colon L_{\diamondsuit}^2 (\partial\Omega)
\rightarrow L_{\diamondsuit}^2(\partial \Omega)$, $g\mapsto f$, which
associates the electric current with the resulting voltage is the so-called 
\textit{Neumann-to-Dirichlet} map (in short NtD). The \textit{forward
operator} associated with EIT is defined by 
\begin{equation}  \label{F}
{\cal F}(\gamma) \ = \ \Lambda _{\gamma} \,,
\end{equation}
with ${\cal F}\colon L_{+}^{\infty}(\Omega) \subset L^\infty(\Omega) \rightarrow {%
\mathcal{L}}(L_{\diamondsuit}^2(\partial\Omega),
L_{\diamondsuit}^2(\partial\Omega))$. The EIT inverse problem consists in
finding $\gamma$ in above equation for a given $\Lambda_\gamma$. However, in
practical situations, only a part of the data can be observed and therefore
the NtD map is not completely available. One has to apply $d \in \mathbb{N}$
currents $g_j \in L_{\diamondsuit}^2(\partial\Omega)$, $j = 1,\dots , d$,
and then record the resulting voltages $f_j = \Lambda_{\gamma} g_j$. We thus
fix the vector $(g_1, \dots , g_d) \in
(L_{\diamondsuit}^2(\partial\Omega))^d$ and introduce the operator $F
\colon L_{+}^\infty(\Omega) \subset L^\infty(\Omega) \rightarrow
(L_{\diamondsuit}^2(\partial\Omega ))^d$, $\gamma \mapsto (\Lambda_\gamma
g_1, \dots, \Lambda_\gamma g_d )$, which is Fr\'echet-differentiable%
\footnote{%
Equipped with an inner product defined in a very natural way, induced by the
inner product in $L^2(\partial\Omega)$, the space $(L_{\diamondsuit}^2(%
\partial\Omega))^d$ is a Hilbert space.} (see, e.g., \cite{rieder2}).

Since an analytical solution of \eqref{eit} is not available in general, the
inverse problem needs to be solved with help of a computer. For this reason,
we construct a triangulation for $\Omega $, $\mathcal{T}=\{T_{i}:i=1,\dots
,M\}$, with $M=1476$ triangles (see the third picture in Figure~\ref{figura0}%
) and approximate $\gamma $ by piecewise constant conductivities: define the
finite dimensional space
$V := (\mathrm{span}\{\chi _{T_{1}},\dots ,\chi_{T_{M}}\} , \norm{\cdot}_{L^{2}(\Omega )} )$.%
\footnote{Notice that $\mathrm{span}\{\chi_{T_{1}},\dots ,\chi_{T_{M}}\} \subset L^{\infty}(\Omega)$.}
We now search the
conductivity in $V$, which means that our reconstructions always have the
form $\sum_{i=1}^{M}\theta _{i}\chi _{T_{i}}$, with $(\theta _{1},\dots
,\theta _{M})\in \mathbb{R}^{M}$. With this new framework our forward
operator reads 
\begin{equation} \label{FG}
F \colon \widetilde{V}\subset V \rightarrow (L^{2}(\partial \Omega))^{d} ,
\quad \gamma \mapsto (\Lambda _{\gamma }g_{1},\dots ,\Lambda _{\gamma
}g_{d}) \, ,
\end{equation}%
where $\widetilde{V} = L_{+}^{\infty}(\Omega) \cap V$.

It is still unclear whether the forward operator associated with the
continuous model of EIT, defined in \eqref{F}, satisfies the tangential cone
condition \eqref{eq:w-tcc}, but the version presented in the restricted set %
\eqref{FG} guarantees this result, at least in a small ball around a
solution, see \cite{rieder2}. The Fr\'{e}chet derivative of $F$, $%
F'\colon \mathrm{int}(\widetilde{V})\rightarrow \mathcal{L}%
(V,(L^{2}(\partial \Omega ))^{d})$, satisfies $F'(\gamma
)h=(w_{1}|_{\partial \Omega },\dots ,w_{d}|_{\partial \Omega })$, where $%
w_{j}\in H_{\diamondsuit }^{1}(\Omega )$ is the unique solution of 
\begin{equation}
\int_{\Omega }\gamma \nabla w_{j}\nabla \varphi =-\int_{\Omega }h\nabla
u_{j}\nabla \varphi \ \text{ for all }\varphi \in H_{\diamondsuit
}^{1}(\Omega )\,,  \label{eit'}
\end{equation}%
with $u_{j}$ solving \eqref{eit} for $g=g_{j}$. The adjoint operator $%
F'(\gamma )^{\ast }\colon (L^{2}(\partial \Omega
))^{d}\rightarrow V$ is given by 
\begin{equation}
F'(\gamma )^{\ast }z=-\sum_{j=1}^{d}\nabla u_{j}\nabla \psi
_{z_{j}}\,,  \label{adjoint}
\end{equation}%
where $z:=(z_{1},\dots ,z_{d})\in (L^{2}(\partial \Omega ))^{d}$ and for
each $j=1,\dots ,d$, the vectors $u_{j}$ and $\psi _{z_{j}}$ are the unique
solutions of \eqref{eit} for $g=g_{j}$ and $g=z_{j}$ respectively.

In our numerical simulations we define $\Omega :=(0,1)\times (0,1)$ and
supply the current-vector $(g_1, \dots , g_d)$ with $d=8$ independent
currents: identifing the
faces of $\Omega $ with the numbers $m=0,1,2,3$, we apply the currents 
\begin{equation*}
g_{2m+k}(x)=%
\begin{cases}
\cos (2k\pi x) & :\ \text{on the face }m \\[1mm] 
\qquad \qquad 0 & :\ \text{elsewhere on }\partial \Omega%
\end{cases}%
\end{equation*}%
for $k=1,2$. The exact solution $\gamma ^{+}$ consists of a constant
background conductivity $1$ and an inclusion $B\subset \Omega $ with
conductivity $2$: 
\begin{equation*}
\gamma ^{+}(x):=%
\begin{cases}
\ 2 & :\ x\in B \\ 
\ 1 & :\ \text{otherwise}%
\end{cases}%
.
\end{equation*}%
The set $B$ models two balls with radii equal $0.15$ and center at the
points $(0.35,0.35)$ and $(0.65,0.65)$. The data,%
\begin{equation}
y:=(\Lambda _{\gamma ^{+}}g_{1},\dots ,\Lambda _{\gamma ^{+}}g_{d}),
\label{data}
\end{equation}%
corresponding to the exact solution $\gamma ^{+}$ are computed using the
Finite Element Method (FEM). The problems \eqref{eit} and \eqref{eit'} have
been solved by FEM as well, but using a much coarser discretization mesh
than the one used to generate the data for avoiding inverse crimes, see
Figure~\ref{figura0}.

\begin{figure}[tb]
\vspace*{-4mm} 
\begin{tabular}{lll}
\parbox[t]{5cm}{\hspace*{-4mm}\includegraphics[height=4cm, trim = 4cm 9cm 4cm 9cm,clip]{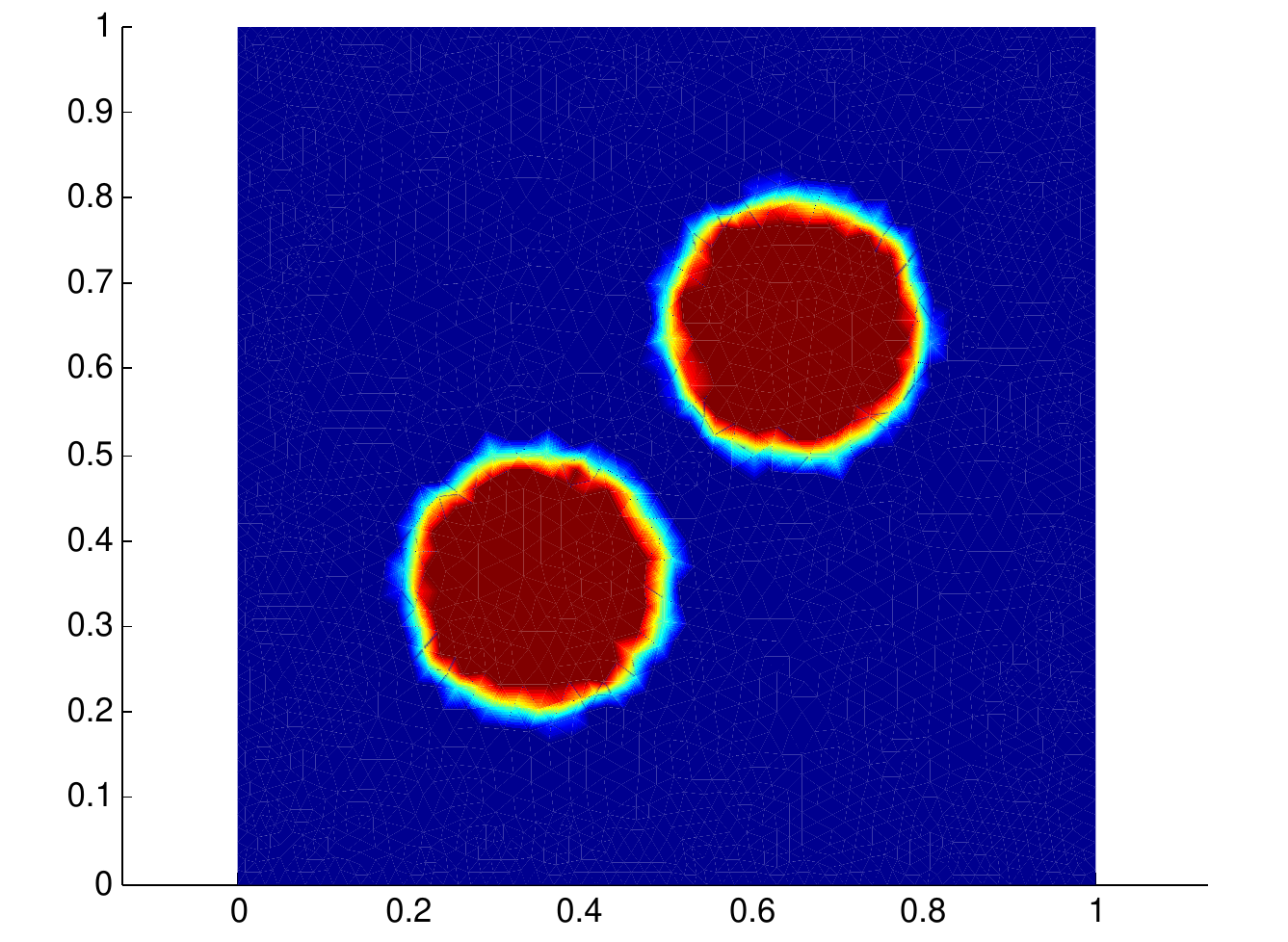}
} & 
\parbox[t]{5cm}{\hspace*{-4mm}\includegraphics[height=4cm, trim = 4cm 9cm 4cm 9cm,clip]{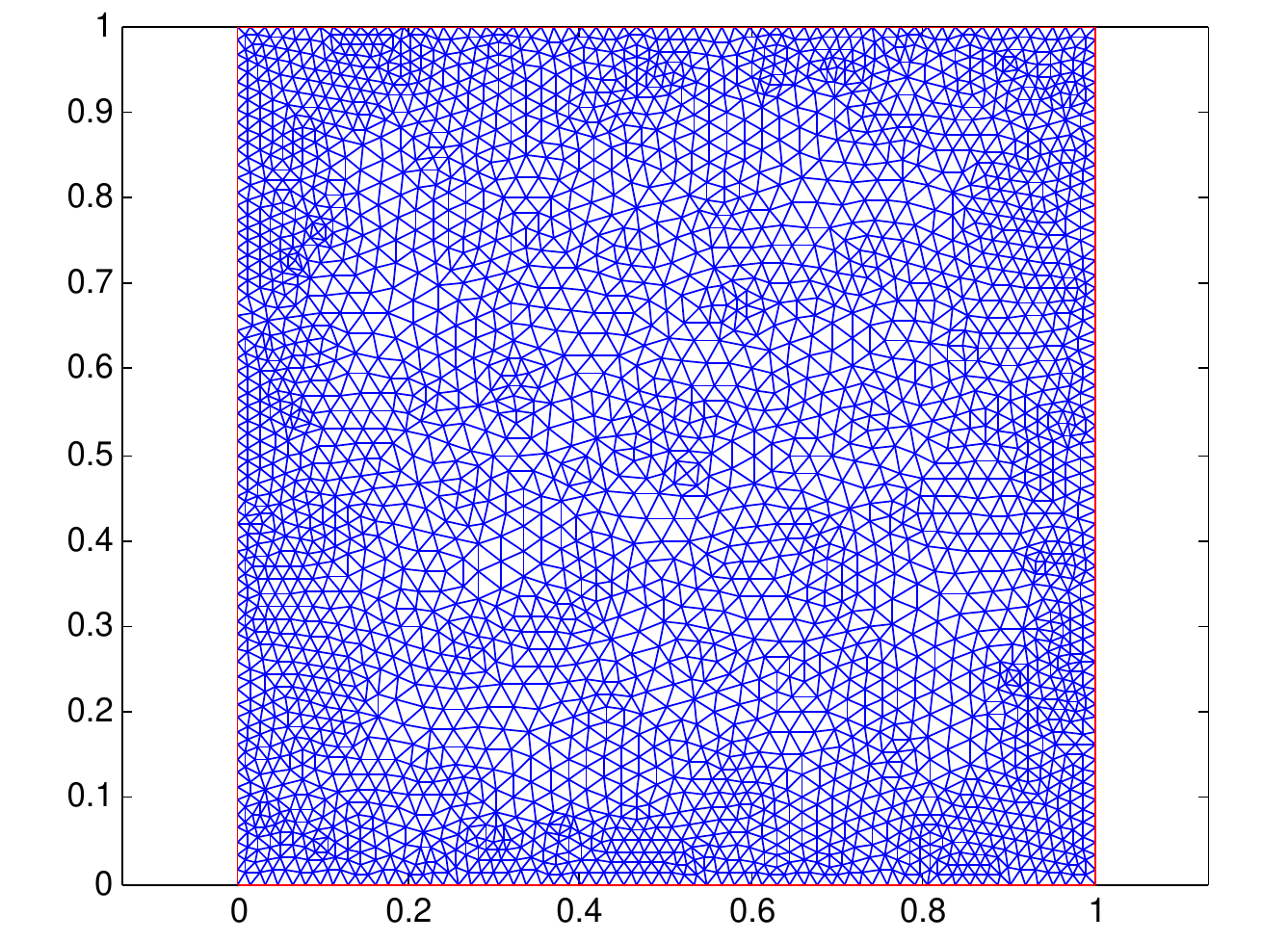}
} & 
\parbox[t]{5cm}{\hspace*{-4mm}\includegraphics[height=4cm, trim = 4cm 9cm 4cm 9cm,clip]{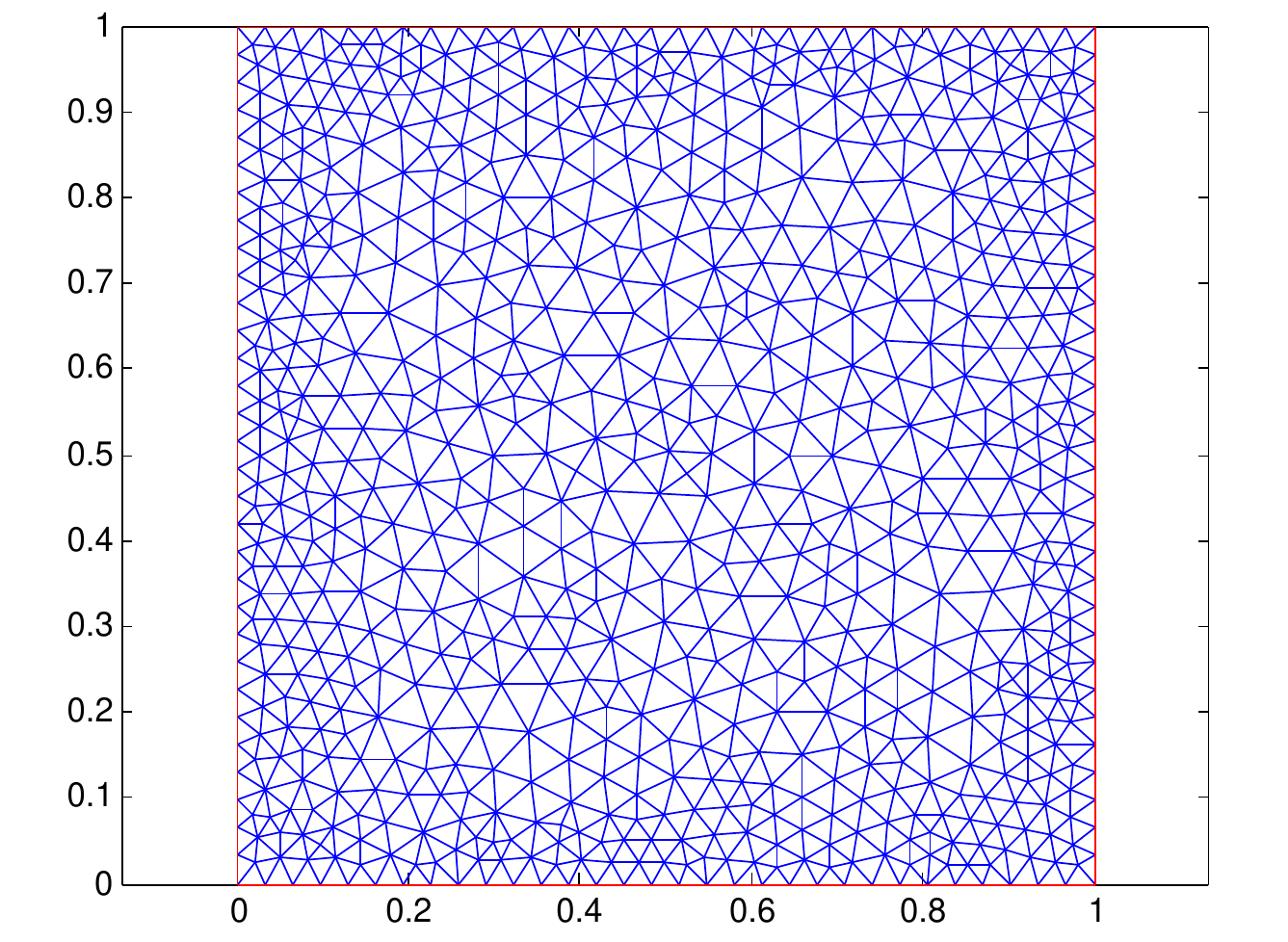}
}%
\end{tabular}%
\caption{Left: Sought solution. Middle: Mesh used to generate the data.
Right: Mesh used to solve the inverse problem.}
\label{figura0}
\end{figure}

It is well known that in this specific problem, undesirable instability
effects may arise from an unfavorable selection of the geometry of the mesh.
For avoiding this problem we employ a strategy using a \textit{%
weight-function} $\omega \colon \Omega \rightarrow \mathbb{R}$ to define the
weighted-space $L_{\omega }^{2}(\Omega ):=\{f\colon \Omega \rightarrow 
\mathbb{R}:\int_{\Omega }|f|^{2}\omega <\infty \}$. This alteration changes
the evaluation of the adjoint operator \eqref{adjoint} in the discretized
setting, see \cite{winkler} and \cite[Subsection $5.1.2$]{fabio3} for
details. In the mentioned references, the authors use the weight-function 
\begin{equation*}
\omega :=\sum_{i=1}^{M}\beta _{i}\chi _{T_{i}}\text{ \ \ \ with \ \ \ }\beta
_{i}:=\frac{\left\Vert F'(\gamma _{0})\chi _{T_{i}}\right\Vert
_{(L^{2}(\partial \Omega ))^{d}}}{|T_{i}|},
\end{equation*}%
where $|T_{i}|$ is the area of triangle $T_{i}$, and the initial iterate
$\gamma_0$ is the constant 1 function.

In the notation of Section~\ref{sec:intro} we have
$F: D(F) \subset X := ( \mathrm{span}\{\chi_{T_{1}},\dots ,\chi_{T_{M}}\} ,\norm{\cdot}_{L_{\omega}^2(\Omega)})
\to (L^{2}(\partial \Omega))^{d} =: Y$, where $D(F) = X \cap L_{+}^\infty(\Omega)$.
We define the \textit{relative error} in the $k-$th iterate $\gamma _{k}$ as 
\begin{equation} \label{it_er}
E_{k}:=100\,\frac{\norm{\gamma_k - \gamma^+}_X}{\norm{\gamma^+}_X} ,
\end{equation}
and use it to compare the quality of the reconstructions.
Finally, we corrupt the simulated data $y$ in \eqref{data} by adding
artificially generated random noise, with a \textit{relative} noise
level $\delta >0$,
\begin{equation} \label{rel_noise}
y^{\delta }=y+\delta \,\mathrm{noi} \left\Vert y\right\Vert_{Y}, 
\end{equation}
where $\mathrm{noi} \in Y$ is a uniformly distributed random variable such that $%
\left\Vert \mathrm{noi} \right\Vert_Y = 1$.

\subsection{Implementation of the range-relaxed Levenberg-Marquardt method} \label{ssec:4.2}

Now, we turn to the problem of finding a pair $(\alpha_k > 0 , h_k \in X)$ in
accordance to Step~[3.1] of Algorithm~I.

An usual choice for the parameters $\alpha_k$ is of geometric type, i.e.,
the parameters are defined \textit{a priori} by the rule $\alpha_k = r\alpha_{k-1}$,
where $\alpha_0 > 0$ and $0<r<1$ (the \textit{decreasing ratio}) are given.
This method is usually very efficient if a good guess for the constant $r$
is available. However, big troubles may arise if the decreasing ratio $r$
is chosen either too large or too small.
Indeed, on the one hand, if the constant $r$ is too large ($r \approx 1$),
then the method becomes slow and the computational costs increase considerably;
on the other hand, the Levenberg-Marquardt method becomes unstable in case $r$
is chosen too small ($r\approx 0$), see Figure~\ref{figura0.5} below.

\begin{figure}[h]
\vspace*{4mm} 
\begin{tabular}{l}
\parbox[t]{8cm}{\hspace*{-2mm}\includegraphics[height=4.8cm, trim = 0cm 7cm 1cm 8cm,clip]{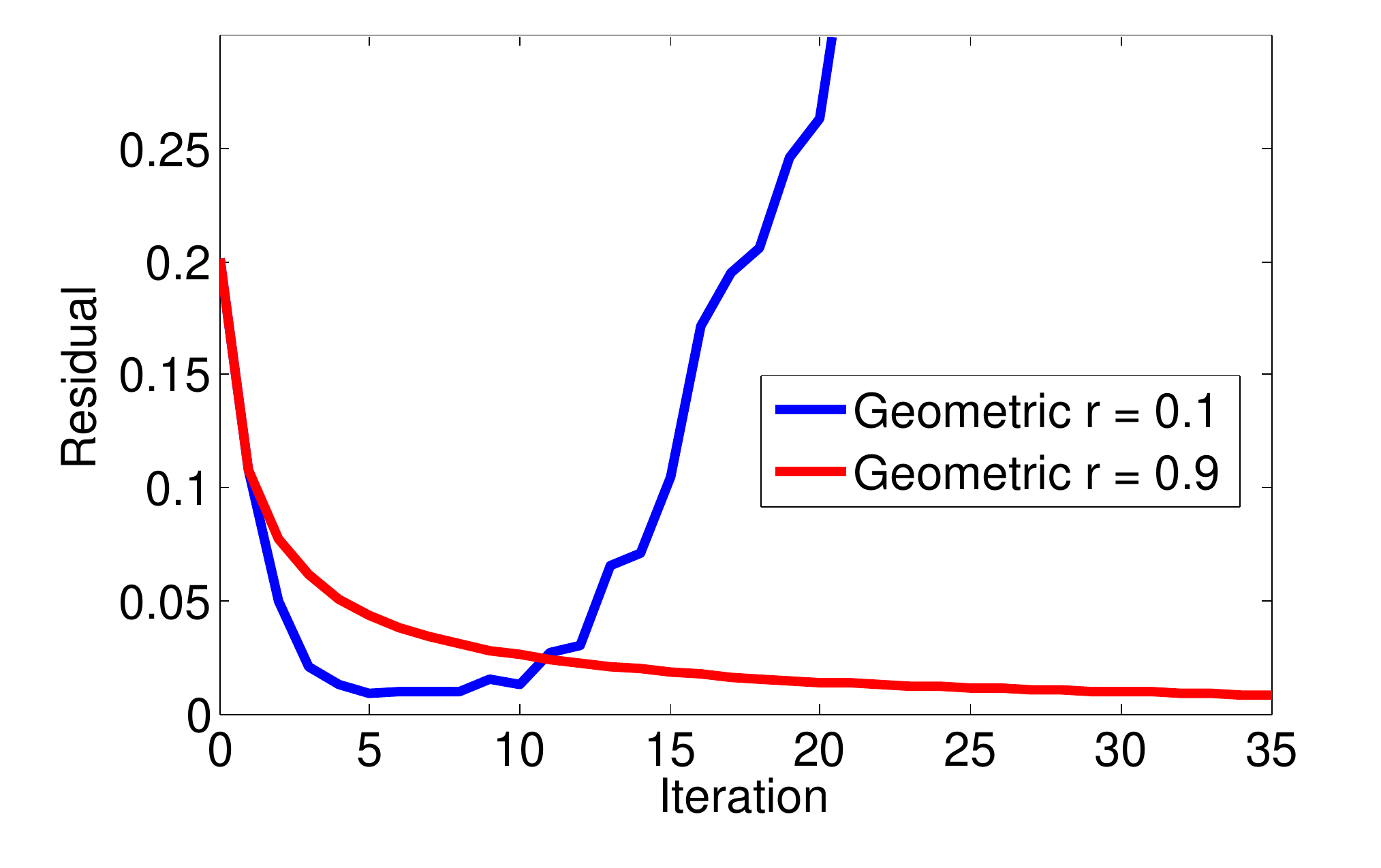}\\
}
\parbox[t]{8cm}{\hspace*{-2mm}\includegraphics[height=4.8cm, trim = 0cm 7cm 1cm 8cm,clip]{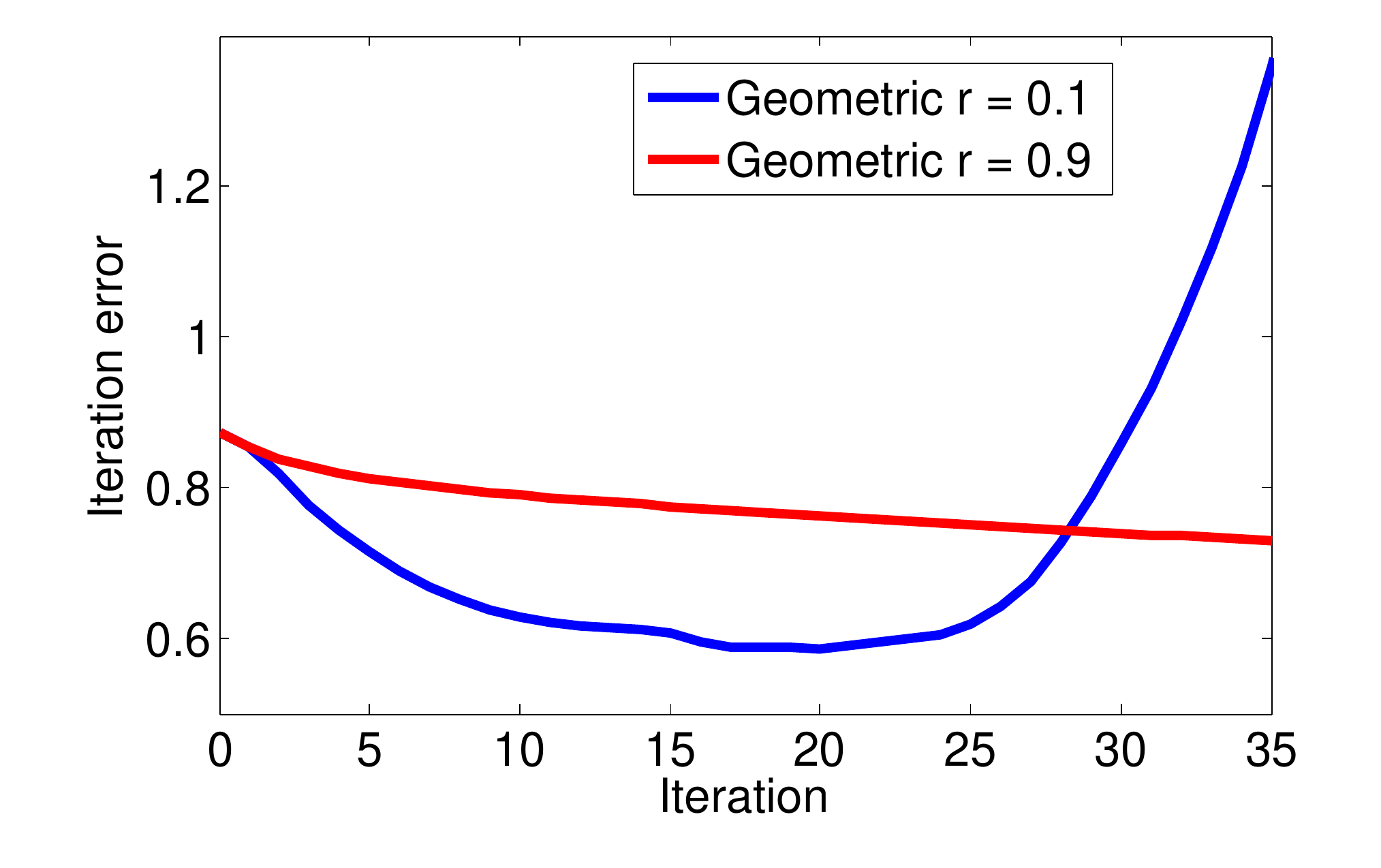}\\
}
\end{tabular}
\caption{Geometric choice of the parameters $\alpha_k$ for $r = 0.1$ and $r = 0.9$, with
$\alpha_0 = 2$.
Noise level $\delta = 0.1 \%$, $\eta = 0.4$ and $\tau = 1.3(1+\eta)/(1-\eta)$.
LEFT: residual, RIGHT: iteration error.}
\label{figura0.5}
\end{figure}

Notice that the $\alpha_k$ defined by the geometric choice does not necessarily
satisfy the problem in Step~[3.1].
We propose a strategy for choosing the decreasing ratio $r$ in each step,
so that the resulting parameter $\alpha_k$ (and the corresponding $h_k$)
are in agreement with Step~[3.1].
For the actual computation of the ratio $r$ in the current step, we use information
on the current iteration and past iterations as well. This is described in the sequel.

We adopt the notation
\begin{equation*}
H_k(\alpha) = \Vert y^{\delta} - F(\gamma_{k}) - F'(\gamma_{k}) h_{\alpha} \Vert ,
\ \alpha > 0 \, ,
\end{equation*}%
where $h_{\alpha}$ is given by
\begin{equation} \label{Tikhonov}
h_\alpha = \big( F'(\gamma_k)^* F'(\gamma_k) + \alpha I \big)^{-1}
             F'(\gamma_k)^* (y^\delta - F(\gamma_k)) \, .
\end{equation}
According to Step~[3.1] in Algorithm~I, we need to determine $\alpha_ k > 0$
such that $H_k(\alpha_k) \in [c_k, d_k]$, where $c_k$ and $d_k$ are
defined in \eqref{eq:ck} and \eqref{eq:dk} respectively.
For doing that, we have employed the \textit{adaptive strategy}
introduced in \cite{MML19}. This algorithm is based on the geometric method
but allows adaptation of the decreasing ratio using {\em a posteriori} information.
First, we define the constants
\begin{equation} \label{intervalo}
\widehat{c}_k = p_1 c_k + ( 1 - p_1 ) d_k
\quad \text{ and } \quad
\widehat{d}_k = p_2 c_k + ( 1 - p_2 ) d_k,
\end{equation}
where $0 < p_1 < p_2 < 1$.
Notice that $[\widehat{c}_k, \widehat{d}_k] \subset [ c_k, d_k ]$.

Choose the initial parameter $\alpha_0 > 0$; compute $h_0 := h_{\alpha_0}$
and $\gamma_1$ according to Algorithm~I.

Choose the initial decreasing ratio $0 < r_0 < 1$, define
$\alpha_1 = r_0 \, \alpha_0$;
compute $h_1 := h_{\alpha_1}$ and $\gamma_2$ according to Algorithm~I.

For $k \geq 1$, we define $\alpha_{k+1} = r_k \alpha_k$, where


\begin{equation} \label{eq:a1a2}
r_k = \left\{ 
\begin{array}{ll}
a_{1}r_{k-1}, & \quad \mathrm{if}\ c_{k-1} \leq H_{k-1}(\alpha _{k-1}) < \widehat{c}_{k-1} \\ 
a_{2}r_{k-1}, & \quad \mathrm{if}\ \widehat{d}_{k-1} < H_{k-1}(\alpha _{k-1}) \leq d_{k-1} \\ 
r_{k-1}, & \quad  \mathrm{if}\ H_{k-1}(\alpha _{k-1}) \in [ \widehat{c}_{k-1} , \widehat{d}_{k-1} ]
\end{array}
\right. .
\end{equation}
Here the constants $0 < a_2 < 1 < a_1$ play the role of correction factors,
and are chosen {\em a priori}.

The idea of the \textit{adaptive strategy} is to observe the behavior of the function
$H_k$ and try to determine how much the parameter $\alpha_k$ should be
decreased in the next iteration.
For example, the number $H_k(\alpha_k)$ lying to the left of the smaller
interval $[\widehat{c}_k, \widehat{d}_k]$ means that $\alpha_k$ was too
small.
We thus multiply the decreasing ratio $r_{k-1}$ by the number $a_1 > 1$, in
order to increase it, and consequently, to decrease the parameter $\alpha_k$
slower than in the previous step, trying to hit $[\widehat{c}_k , \widehat{d}_k]$
in the next iteration.
This algorithm is efficient in terms of computational cost: Like the
geometric choice for $\alpha_k$, it requires only one minimization of a
Tikhonov functional in each iteration.
Further, the \textit{adaptive strategy} has the additional advantage of correcting the
decreasing ratio if this ratio is either too large or too small.

An attentive reader could object that, in some iterations, the evaluated
parameter $\alpha_k$ may lead to a number $H_k(\alpha_k)$ which does
not belong to the interval $[c_k, d_k]$ defined in Step~[3.1].
This is indeed possible!
In this situation, we apply the secant method in order to recalculate
$\alpha_k$ such that $H_k(\alpha_k) \in [c_k, d_k]$, before
starting the next iteration.
This is however an expensive task, since each step of the secant
method demands the additional minimizations of Tikhonov functionals.

It is worth noticing that this situation has been barely observed in our
numerical experiments, occurring only in the cases when either the initial
decreasing ratio $r_0$ or the initial guess $\alpha_0$ are poorly chosen.

\subsection{Numerical realizations} \label{ssec:4.3}

For the constant $\tau$ in \eqref{def:tau-sigma} we use $\tau = 1.3( 1+\eta ) / (1-\eta)$,
where $\eta = 0.4$ is the constant in (A2).
Moreover, we choose $p = 0.1$ and $\varepsilon = 0.1 [\tau(1-\eta)-(1+\eta)] / \eta\tau$
in \eqref{def:tau-sigma}.
The constants in \eqref{intervalo} are $p_{1} = 1/3$ and $p_{2} = 2/3$, while
the constants in \eqref{eq:a1a2} are $a_{1} = 2$ and $a_{2} = 1/2$.
\bigskip

\noindent
{\bf First test (one level of noise):} \\
The goal of this test is to investigate the performance of our rrLM method with adaptive
strategy (a posteriori) for computing the parameters, with respect of different choices
of initial decreasing ratio $r_0$.

As observed in Figure~\ref{figura0.5}, the performance of the LM method with geometric
choice (a priori) of parameters is very sensitive to the choice of the (constant)
decreasing ratio $r < 1$.

We implement the rrLM method (using \textit{adaptive strategy}) with
$r_0 = 0.1$ and $r_0 = 0.9$. In Figure~\ref{figura1} the results of the the rrLM method
are compared with the LM method using geometric choice of parameters
(see {\color{black} top-left, top-right and bottom-left} pictures).

--- [GREEN] rrLM with $r_0 = 0.1$, reaches discrepancy with $k^* = 11$ steps;

--- [MAGENTA] rrLM with $r_0 = 0.9$, reaches discrepancy with $k^* = 11$ steps;

--- [RED] LM with $r = 0.9$, reaches discrepancy with $k^* = 36$ steps;

--- [BLUE] LM with $r = 0.1$, does not reach discrepancy. \\
The noise level is $\delta = 0.1\%$. All methods are started with $\alpha_0 = 2$.
The last picture in Figure~\ref{figura1} ({\color{black} bottom-right}) shows the
values of the linearized residual $H_k(\alpha_k)$ as well as the intervals
$[c_k,d_k]$ (see \eqref{eq:ck} and \eqref{eq:dk}) for the rrLM with $r_0 = 0.9$.

\begin{figure}[h]
\vspace{-3cm}
\centerline{\includegraphics[width=0.5\textwidth]{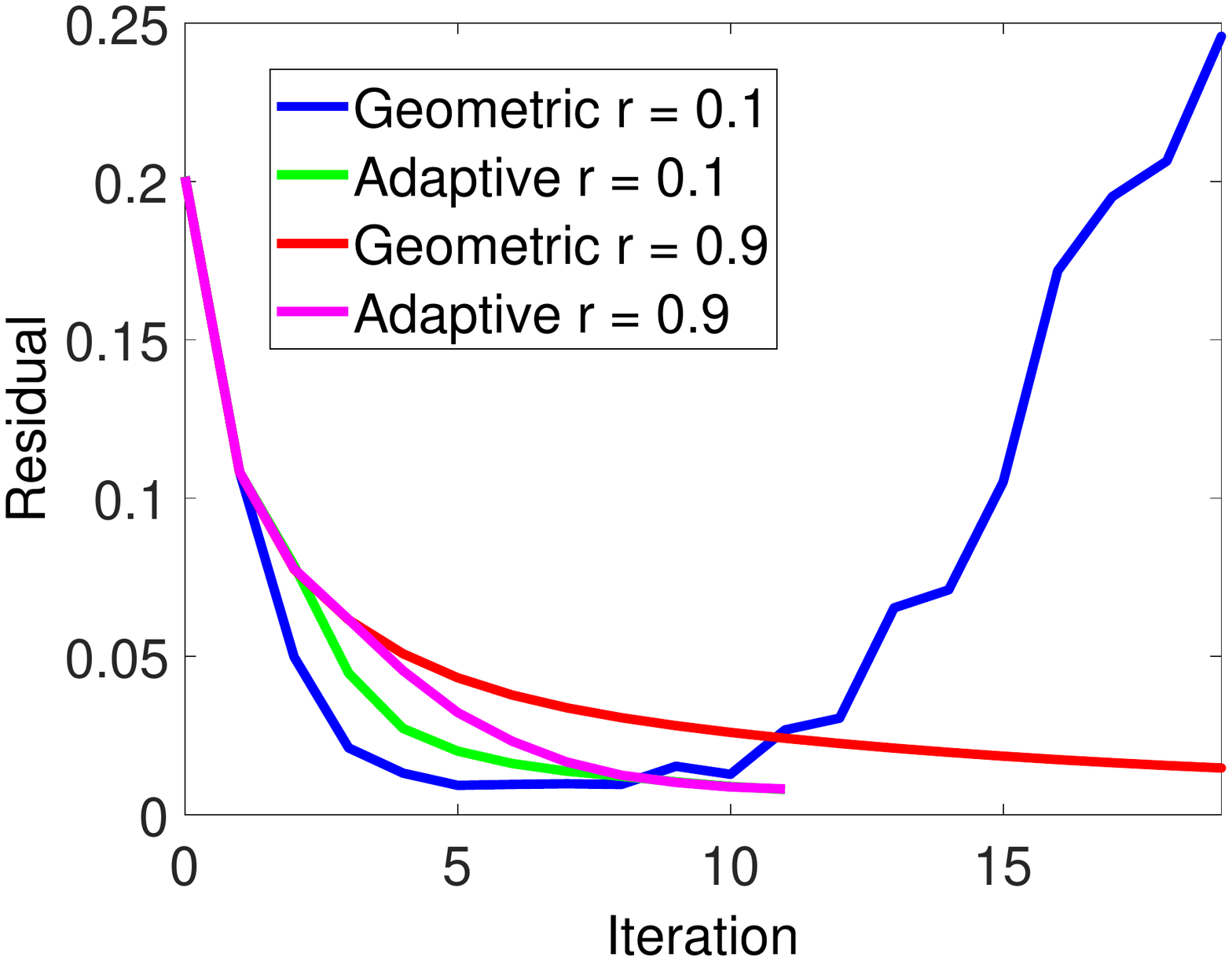}
            \includegraphics[width=0.5\textwidth]{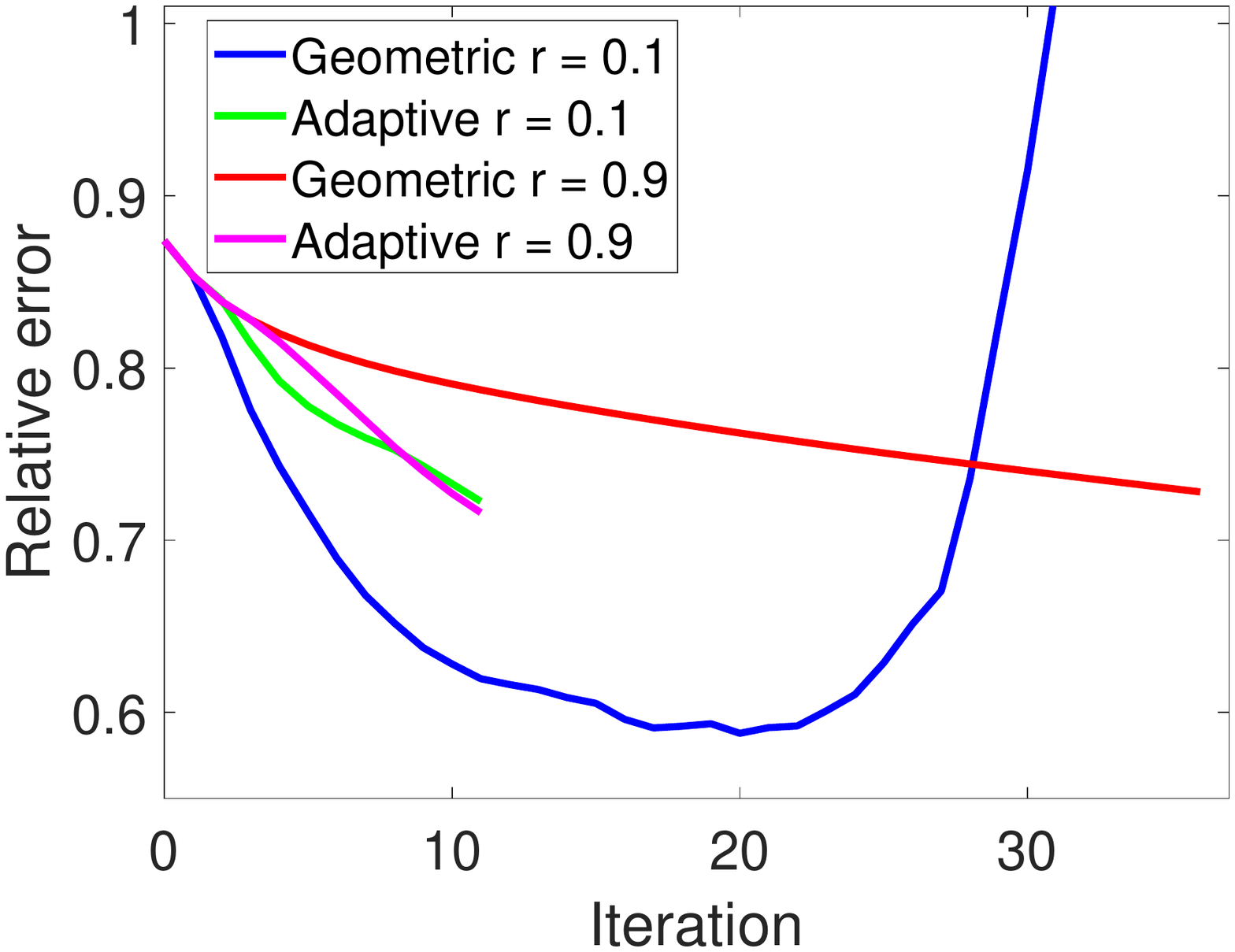}
}
\vspace{-4.5cm}
\centerline{\includegraphics[width=0.5\textwidth]{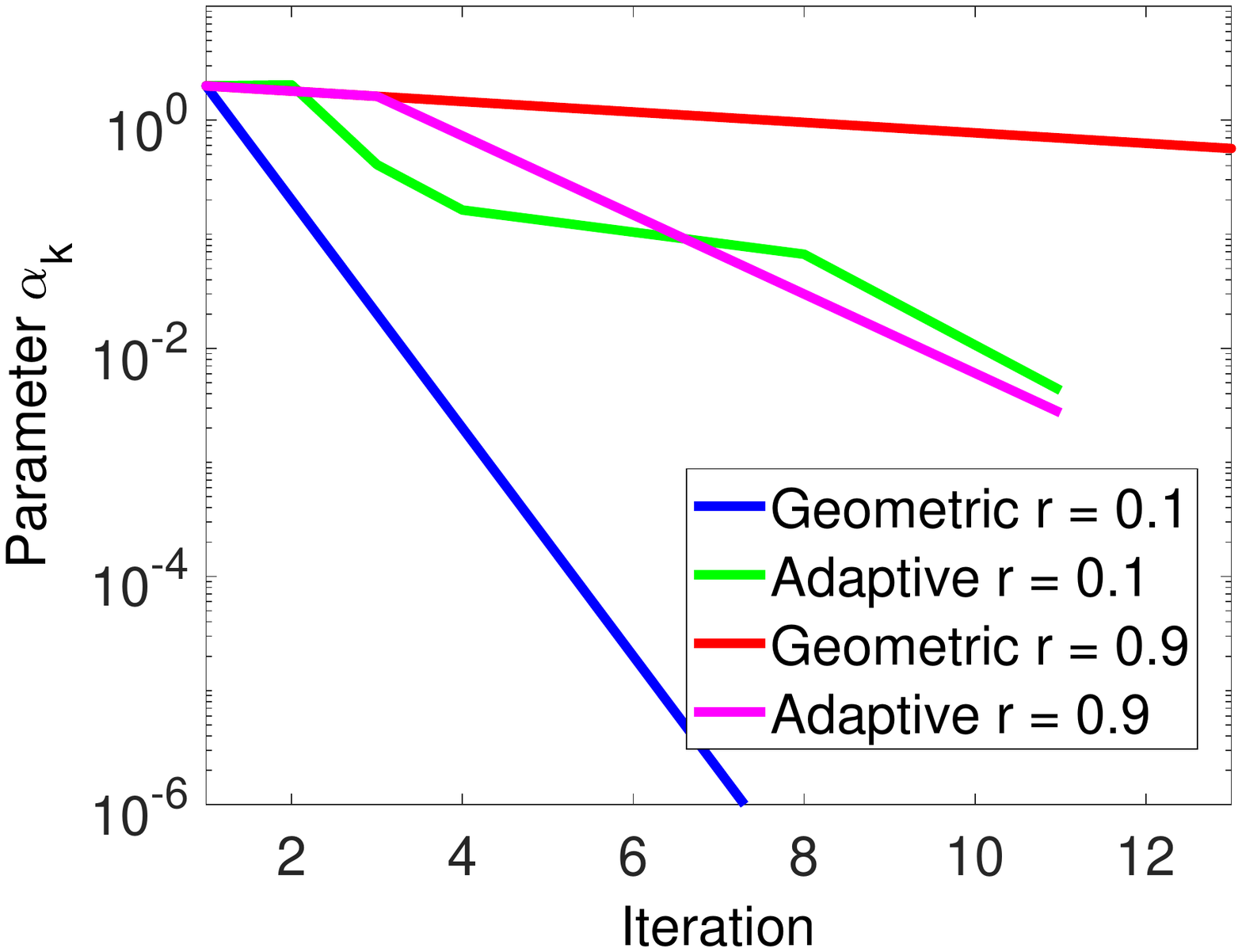}
            \includegraphics[width=0.5\textwidth]{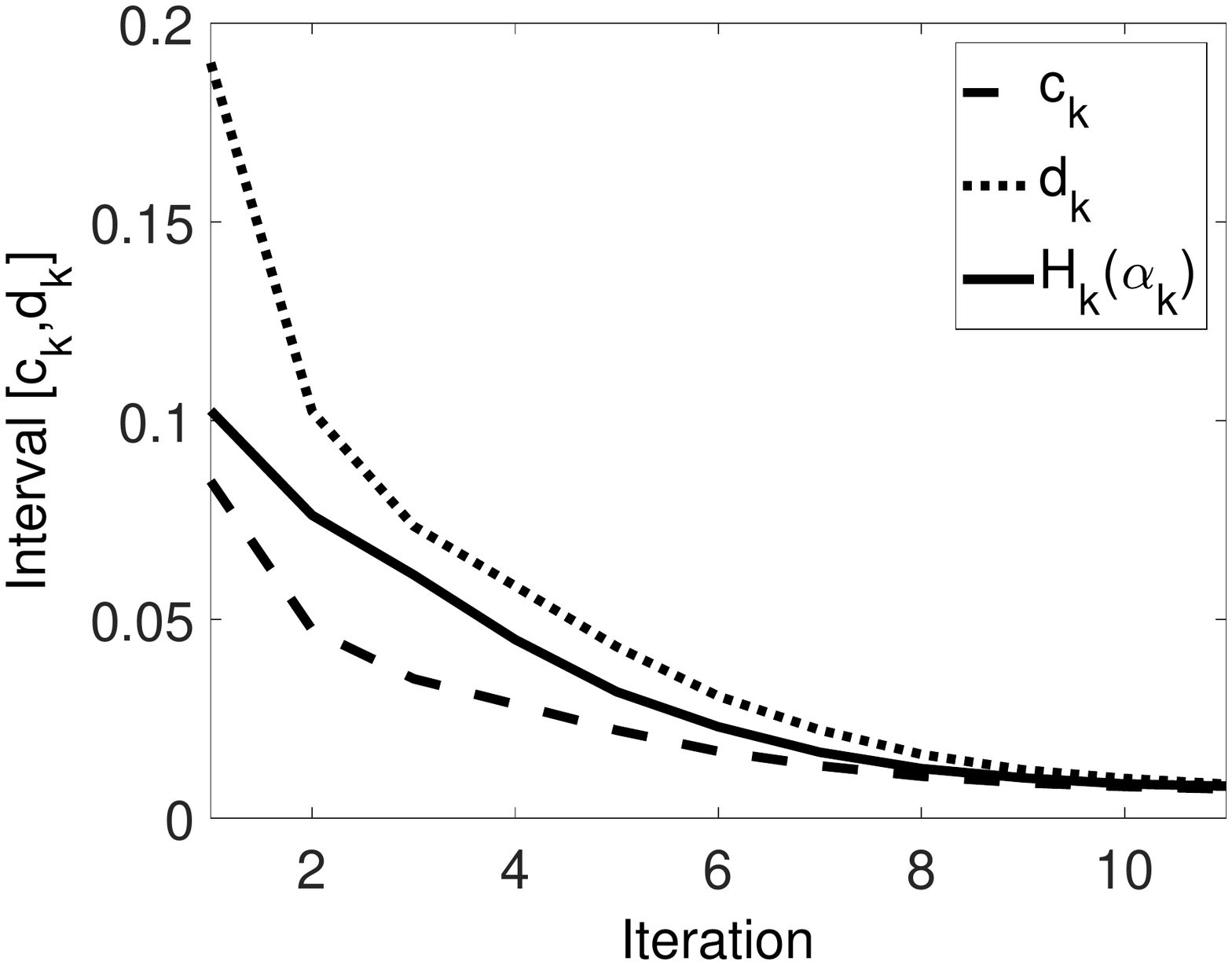}
}
\vspace{-2.5cm}
\caption{First test: Noisy data, $\delta = 0.1\%$. {\color{black} TOP-LEFT}: Residual.
{\color{black} TOP-RIGHT}: Relative iteration error.
{\color{black} BOTTOM-LEFT}: Parameter $\protect\alpha _{k}$.
{\color{black} BOTTOM-RIGHT}: Linearized residual $H_k(\alpha_k)$ and the numbers $c_{k}$ and $d_{k}$ for the rrLM with $r_0 = 0.9$.}
\label{figura1}
\end{figure}

From this first test we draw the folllowing conclusions: \\
$\bullet$ The rrLM method (using \textit{adaptive strategy}) is robust
with respect of the choice of the (initial) decreasing ratio.
We tested two poor choices of initial decreasing ratios (namely $r_0 = 0.1$ and $r_0 = 0.9$);
nevertheless the performance of the rrLM method in both cases is stable and numericaly
efficient.
{\color{black} For rrLM method, the relative error obtained for $r_0 = 0.1$ is comparable
to that obtained for $r_0 = 0.9$ (see top-right picture in Figure~\ref{figura1}).} \\
$\bullet$ We also tested the rrLM method (using \textit{adaptive strategy})
and the LM method (using geometric choice of parameters) for $r = r_0 = 0.5$, which seems
to be the "optimal" choice of constant decreasing ratio.
In this case, both methods performed similarly.
Moreover, the performance of the rrLM method (number of iterations and numerical effort)
was similar to the ones depicted in Figure~\ref{figura1} using $r_0 = 0.1$ and $r_0 = 0.9$. \\
$\bullet$ The rrLM method (using \textit{adaptive strategy}) "corrects"
eventual poor choices of the decreasing ratio. If $r_0$ is too small, the \textit{adaptive strategy}
increases this ratio during the first iterations (GREEN curve in Figure~\ref{figura1})
preventing instabilities (compare with the LM method using geometric choice of parameters ---
BLUE curve in Figure~\ref{figura1}). \\
On the other hand, if $r_0$ is large (close to one), the \textit{adaptive strategy} decreases this
ratio during the first iterations (MAGENTA curve in Figure~\ref{figura1}), preventing slow
convergence (compare with the LM method using geometric choice of parameters --- RED curve
in Figure~\ref{figura1}). \\
$\bullet$ The last picture in Figure~\ref{figura1} ({\color{black} bottom-right}) shows that
the linearized residual $H_k(\alpha_k)$, computed using the \textit{adaptive strategy},
satisfies \eqref{eq:rb} in Step~[3.1] of Algorithm~I.
Consequently, this strategy provides a numerical realization of Algorithm~I, which is in
agreement with the theory devised in this article.
\bigskip

\noindent
{\bf Second test (several levels of noise):} \\
The goal of this test is twofold: (1st)
We validate the regularization property (see Theorem~\ref{th:reg} and Corollary~\ref{cor:reg})
by choosing different levels of noise $\delta > 0$, and observing what happens when the noise
level decreases;
(2nd) We compare the numerical effort of the rrLM method (with \textit{adaptive strategy})
with the LM method (with geometric choice of parameters).

In what follows we present a set of experiments with four different levels of noise
$\delta > 0$ namely, $\delta = 0.8\%$, $\delta =0.4\%$, $\delta =0.2\%$, $\delta = 0.1\%$.
In each scenario above, we implemented the rrLM method (with \textit{adaptive strategy})
as well as the LM method (with geometric choice of parameters).
 
{\color{black}
For the implementation of the LM method with geometric choice of parameters we use
the constant decreasing ratios: $r_0 = 0.9$, $r_0 = 0.5$ and $r_0 = 0.1$.
For the implementation of the rrLM method we used the same choices of $r_0$ as starting
value for $r$ together with the \textit{adaptive strategy}.
In all implementations $\alpha_0 = 2$ is used.
Comparisons of these methods are presented in Tables~\ref{table1} and~\ref{table2}.
Three distinct indicators are used, namely}
\\
\indent -- Number of iterations to reach discrepancy $k^* = k^*(\delta)$ (see Step~[3.3]);
\\
\indent -- Total number of Tikhonov functionals minimized for $k = 0, \dots k^*-1$, denoted by $N_{k^*}$;%
    \footnote{The numbers $k^*$ and $N_{k^*}$ are always the same in the geometric choice
    (LM method), but $N_{k^*}$ may be larger than $k^*$ in the \textit{adaptive strategy} (rrLM method).}
\\
\indent -- Relative iteration error at step $k = k^*$, denoted by $E_{k^*}$ (see \eqref{it_er}).%
    \footnote{It is worth noticing that the initial iteration error is $E_{0}=87.39\%$ in all
    four scenarios above.}

\begin{table}[t]
\centering
{\color{black}
\begin{tabular}{ccc|cc|cc}
\hline
\multicolumn{7}{c}{$k^*(N_{k^*})$} \\
\hline
       & \multicolumn{2}{c}{($r_0=0.9$)} & \multicolumn{2}{|c|}{($r_0=0.5$)} & \multicolumn{2}{c}{($r_0=0.1$)} \\
$\delta (\%)$ & rrLM & LM  &   rrLM & LM    &  rrLM & LM     \\
\hline\hline
$0.8$ & 5(6)  & 3(3)       & 4(5)  & 3(3)   & 5(8)  & 3(3)   \\
$0.4$ & 8(8)  & 8(8)       & 6(6)  & 4(4)   & 8(12) & 4(4)   \\
$0.2$ & 9(9)  & 18(18)     & 7(7)  & 7(7)   & 8(11) & Fails  \\
$0.1$ & 11(11)& 35(35)     & 10(10)& 10(10) & 11(14)& Fails  \\
\hline
\end{tabular}
}
\caption{Comparison between rrLM and LM methods: Computational effort.}
\label{table1}
\end{table}
\begin{table}[t]
\centering
{\color{black}
\begin{tabular}{ccc|cc|cc}
\hline
\multicolumn{7}{c}{$E_{k^*}$} \\
\hline
       & \multicolumn{2}{c}{($r_0=0.9$)} & \multicolumn{2}{|c|}{($r_0=0.5$)} & \multicolumn{2}{c}{($r_0=0.1$)} \\
$\delta (\%)$
      &   rrLM & LM & rrLM & LM   & rrLM & LM   \\
\hline\hline
$0.8$ & 82.6 & 82.7 & 82.8 & 81.5 & 82.8 &  80.9 \\
$0.4$ & 79.7 & 79.7 & 79.5 & 79.7 & 79.6 &  79.5 \\
$0.2$ & 76.5 & 76.5 & 76.3 & 76.6 & 76.4 & Fails \\
$0.1$ & 71.5 & 72.9 & 71.6 & 71.7 & 72.1 & Fails \\
\hline
\end{tabular}
}
\caption{Comparison between rrLM and LM methods: Relative iterative error at
the final iteration.}
\label{table2}
\end{table}

\noindent From this second test we draw the following conclusions: \\
$\bullet$ \ For both methods $k^*$ increases and $E_{k^*}$ decreases as $\delta$
becomes smaller (validating the regularization property). \\
$\bullet$ \ For each fixed noise level $\delta$, the values of $E_{k^*}$ are similar for both methods. \\
{\color{black}
$\bullet$ \ If the noise level is small ($\delta = 0.1\%$ and $\delta = 0.2\%$),
the rrLM method is more efficient than the LM method for $r_0 = 0.9$.
Both methods perform similarly for $r_0 = 0.5$.
For $r_0 = 0.1$ the LM method fails to converge, while the rrLM method succeed
in reaching the stopping criterium. } \\
{\color{black}
$\bullet$ \ For higer levels of noise ($\delta = 0.4\%$ and $\delta = 0.8\%$),
both methods perform similarly for $r_0 = 0.9$ and $r_0 = 0.5$.
For $r_0 = 0.1$ the LM method converges faster than the rrLM method.
This is due to the fact that rrLM needs to correct the initial guess for
$\alpha_0 = 2$.
} \\
{\color{black}
$\bullet$ \ For levels of noise higher than $0.8\%$, the rrLM stops after 2
or less iterations (for different choices of $r_0$). Consequently, this experiments
do not give relevant information about the performance of our method. } \\
$\bullet$ \ For the rrLM method, the values of $k^*$ and $N_{k^*}$ are identical in most of the
scenarios of Table~\ref{table1}, i.e., only one Tikhonov functional is minimized in each step
(this is the same numerical cost for one step of the LM method with geometric choice of parameters).

The last conclusion validates the \textit{adaptive strategy} for computing the parameters
$\alpha_k$ as an efficient alternative for the numerical implementation of Step~[3.1] in
Algorithm~I.

\section{Final remarks and conclusions} \label{sec:conclusion}

In this article we address the Levenberg-Marquardt method for solving
nonlinear ill-posed problems and propose a novel range-relaxed criteria
for choosing the Lagrange multipliers, namely:
the new iterate is obtained as the projection of the current one onto a
level-set of the linearized residual function; this level belongs to an
interval (or \emph{range}), which is defined by the current nonlinear
residual and by the noise level (see Step [3.1] of Algorithm~I).

The main contributions in this article are:
\\
$\bullet$ \ We derive a complete convergence analysis:
convergence (Theorem~\ref{th:noiseless}),
stability (Theorem~\ref{th:stabil}),
semi-convergence (Theorem~\ref{th:reg}).
We also prove monotonicity of iteration error (Theorem~\ref{th:wd})
and geometric decay of residual (Proposition~\ref{pr:residual_decay}).
Moreover, we prove convergence to minimal-norm solution under additional
null-space condition (\ref{eq:nucleo}), in both exact and noisy data cases.
\\
$\bullet$ \ We give a novel proof for the stability result, which uses
non standard arguments.
In the classical stability proof, since each Lagrange multiplier is
uniquely defined by an (implicit) equation, the set of successors
(Definition~\ref{def:sucessor}) of each $x_k^\delta$ is singleton.
However, due to our range-relaxed criteria \eqref{eq:rb}, each
set of successors may contain infinitely many elements; consequently, the
subsequences $\{ \delta_{j_m} \}_{m\in\N}$ obtained in Theorem~%
\ref{th:stabil} do depend on the iteration index $k$.
\\
$\bullet$ \ We devise a numerical algorithm, based on the \textit{adaptive strategy}
(see Subsection~\ref{ssec:4.2}), for implementing the range-relaxed
criteria proposed in this article. Its main features are: \\[-4.2ex]
\begin{itemize}
\item[--] Efficiency in terms of computational cost: Like the LM with geometric
({\em a priori}) choice of parameters, it (almost always) requires only
one minimization of a Tikhonov functional in each iteration. \\[-4.5ex]
\item[--] Correction of the decreasing ratio if this ratio is either too
large or too small. \\[-4.5ex]
\item[--] The computed pairs $( \alpha_k, h_k)$ satisfy \eqref{eq:rb} for
all $k > 0$, i.e., this algorithm provides a numerical realization of
Algorithm~I.
\end{itemize}

\section*{Acknowledgments}

The work of A.L. is supported by the Brazilian National Research Council CNPq,
grant 311087/2017--5 and by the Alexander von Humboldt Foundation AvH.

\bibliographystyle{amsplain}
\bibliography{projectedLM}

\end{document}